\documentclass[a4paper,10pt]{amsart}

\usepackage{amsmath}
\usepackage{amsthm}
\usepackage{amssymb}

\usepackage{mathrsfs}
\usepackage[all]{xy}

\newcommand{\set}[2]{\left\{ #1 \mid #2 \right\}}

\newcommand{\adjunction}[4]{\xymatrix{ #1 \ar@<1ex>[rr]^-{#3} && #2 \ar@<1ex>[ll]^-{#4}}}

\newcommand{\SDR}[5]{\xymatrix{*[r]{#1} \ar@<1ex>[r]^-{#3} \ar@(ul,dl)[]_{#5} & #2 \ar@<1ex>[l]^-{#4}}}
\newcommand{\bigSDR}[5]{\xymatrix{*[r]{#1} \ar@<1ex>[rr]^-{#3} \ar@(ul,dl)[]_{#5} && #2 \ar@<1ex>[ll]^-{#4}}}
\newcommand{\bigbigSDR}[5]{\xymatrix{*[r]{#1} \ar@<1ex>[rrr]^-{#3} \ar@(ul,dl)[]_{#5} &&& #2 \ar@<1ex>[lll]^-{#4}}}

\newcommand{\Sing}{\operatorname{Sing}}

\newcommand{\HH}{\operatorname{H}}
\newcommand{\rHH}{\widetilde{\HH}{}^*}
\newcommand{\tensor}{\otimes}
\newcommand{\Hom}{\operatorname{Hom}}
\newcommand{\CC}{\mathbb{C}}
\newcommand{\QQ}{\mathbb{Q}}
\newcommand{\RR}{\mathbb{R}}
\newcommand{\ZZ}{\mathbb{Z}}
\newcommand{\LL}{\mathbb{L}}
\newcommand{\BB}{B}

\newcommand{\Aut}{\operatorname{Aut}}
\newcommand{\tAut}{\widetilde{\Aut}}
\newcommand{\aut}{\operatorname{aut}}

\newcommand{\Map}{\operatorname{Map}}
\newcommand{\Emb}{\operatorname{Emb}}
\newcommand{\Top}{Top}

\newcommand{\Bun}{\operatorname{Bun}}

\newcommand{\rank}{\operatorname{rank}}

\newcommand{\MC}{\mathsf{MC}}
\newcommand{\gl}{\mathfrak{g}}

\newcommand{\im}{\operatorname{im}}
\newcommand{\spann}{\operatorname{span}}

\newcommand{\colim}{\operatorname{colim}}
\newcommand{\hocolim}{\operatorname{hocolim}}

\newcommand{\MT}{MT}
\newcommand{\Diff}{\operatorname{Diff}}
\newcommand{\tDiff}{\widetilde{\Diff}}
\newcommand{\Wh}{\operatorname{Wh}}
\newcommand{\A}{A}
\newcommand{\K}{K}
\newcommand{\Grass}{\operatorname{Grass}}
\newcommand{\Th}{\operatorname{Th}}
\newcommand{\Spin}{\operatorname{Spin}}
\newcommand{\Ss}{\mathcal{S}}
\newcommand{\Ssp}{\underline{\underline{\mathcal{S}}}}
\newcommand{\Lsp}{\mathbb{L}}
\newcommand{\Fr}{\operatorname{Fr}}

\newcommand{\CP}[1]{\CC \hspace{-2.5pt} \operatorname{P}^{#1}}

\newcommand{\interior}{\operatorname{int}}

\newcommand{\Ext}{\operatorname{Ext}}

\newcommand{\Sp}{\operatorname{Sp}}
\newcommand{\SO}{\operatorname{SO}}
\newcommand{\OO}{\operatorname{O}}
\newcommand{\GL}{\operatorname{GL}}

\newcommand{\dquot}{/\hspace{-3pt}/}

\newtheorem{theorem}{Theorem}[section]
\newtheorem{proposition}[theorem]{Proposition}

\newtheorem{lemma}[theorem]{Lemma}

\theoremstyle{definition}

\newtheorem{remark}[theorem]{Remark}

\title{Homological stability of diffeomorphism groups}
\author{Alexander Berglund and Ib Madsen}
\thanks{The first author was supported by the Danish National Research Foundation (DNRF) through the Centre for Symmetry and Deformation. The second author was supported by ERC adv grant no. 228082.}

\begin{document}

\begin{abstract}
In this paper we prove a stability theorem for block diffeomorphisms of $2d$-dimensional manifolds that are connected sums of $S^d\times S^d$. Combining this with a recent theorem of S. Galatius and O. Randal-Williams and Morlet's lemma of disjunction, we determine the homology of the classifying space of their diffeomorphism groups relative to an embedded disk in a stable range.
\end{abstract}

\maketitle

\begin{center} {\bf Dedicated to Dennis Sullivan on the occasion of his 70th birthday} \end{center}

\section{Introduction} \label{sec:intro}
The traditional method to obtain homotopical and homological information about diffeomorphism groups of high dimensional smooth manifolds is a two step procedure: the surgery exact sequence handles the larger group of block diffeomorphisms and Waldhausen's K-theory of spaces connects block diffeomorphisms and actual diffeomorphisms. In practice one is forced to retreat to rational information. The method was used by Farrell and Hsiang \cite{FH} to evaluate the rational homotopy groups of the orientation preserving diffeomorphism groups of spheres in a modest range of dimensions;
\begin{equation} \label{eq:bdiffsn}
\pi_k(\BB\Diff(S^n))\tensor \QQ = \left\{ \begin{array}{ll} \QQ, & \mbox{$k \equiv 0 \, (4)$, $n$ even} \\ \QQ\oplus \QQ, & \mbox{$k\equiv 0 \, (4)$, $n$ odd} \\ 0, & \mbox{otherwise.} \end{array} \right.
\end{equation}
for $0<k<\frac{n}{6}-7$. The orthogonal group $SO(n+1)$ is a subgroup of $\Diff(S^n)$, and $\pi_kSO(n+1)$ accounts for one $\QQ$ when $k\equiv 0 \,(4)$. The second $\QQ$ in line two above is of a different nature. It comes from the connection between Waldhausen's algebraic K-theory of spaces and the homotopy theory of the homogeneous space $\tDiff(S^n)/\Diff(S^n)$ of block diffeomorphisms modulo diffeomorphisms. This is most clearly expressed in the following theorem from \cite{WW}:
$$\pi_k \tDiff(M^n)/\Diff(M^n) \cong \pi_{k+1} (\Wh(M^n)\dquot \ZZ/2),\quad k \ll n.$$
Here $\Wh(M^n)$ is a factor in Waldhausen's $\A(M)$,
$$\A(M) \cong Q(M_+) \times \Wh(M),\quad Q = \Omega^\infty \Sigma^\infty,$$
cf. \cite{JRW}. For $k<n$,
$$\pi_k \A(S^n)\tensor \QQ \cong \pi_k \A(*)\tensor \QQ \cong \pi_k \K(\ZZ) \tensor \QQ,$$
so one $\QQ$ in the second line of \eqref{eq:bdiffsn} is represented by the Borel regulator. In more geometric terms,
$$\pi_k \BB\Diff(D^n,\partial)\tensor \QQ \cong \pi_{k+1} \K(\ZZ)\tensor \QQ$$
for odd $n\gg k$.

In general, $\A(M)$, and more generally algebraic K-theory of ring spectra, is very hard to evaluate explicitly, although the topological cyclic homology functor of \cite{BHM} gives a description of the homotopy fiber of the ``linearization'' map $\A(M)\rightarrow \K(\ZZ[\pi_1M])$ and thus reduces the calculation of $\A(M)$ to Quillen's space $K(\ZZ[\pi_1 M])$ whose homotopy groups are the higher algebraic K-groups of the integral group ring of the fundamental group of $M$, cf. \cite{M,DGM}.

With the solution of the generalized Mumford conjecture \cite{MW} a new method was introduced to study diffeomorphism groups, based on embedded surfaces and the Pontryagin-Thom collapse map. \cite{MW} only treated surfaces but the viewpoint was generalized in \cite{GMTW} where the homotopy type of the embedded cobordism category was determined. Most recently, Galatius and Randal-Williams used surgery techniques, combined with the results from \cite{GMTW} and the group completion theorem, to determine the ``stable'' homology of diffeomorphism groups of $(d-1)$-connected $2d$-dimensional compact manifolds. We recall their main result. Let
\begin{equation} \label{eq:(2)}
BO(2d)[k,\infty) \stackrel{\theta_k}{\rightarrow} BO(2d)
\end{equation}
be the $(k-1)$-connected cover of $BO(2d)$, i.e., a Serre fibration with
$$
(\theta_k)_*\colon \pi_i(BO(2d)[k,\infty)) \rightarrow \pi_i(BO(2d))
$$
an isomorphism when $i\geq k$ and $\pi_i(BO(2d)[k,\infty)) = 0$ for $i<k$. There is a spectrum $\MT^{\theta_k}(2d)$ associated with $\theta_k$, namely the Thom spectrum of $-\theta_k^*(U_{2d})$ where $U_{2d}$ is the tautological $2d$-dimensional vector bundle over $BO(2d)$. In more detail, $BO(2d)$ is the colimit (union) of the Grassmannians $\Grass_{2d}(\RR^N)$ of $2d$-dimensional linear subspaces of $\RR^N$ as $N\rightarrow \infty$. Let
$$\theta_k\colon B_k(N)\rightarrow \Grass_{2d}(\RR^N)$$
denote the fibration induced from \eqref{eq:(2)}. There are two basic vector bundles over the Grassmannian, the tautological $2d$-dimensional vector bundle $U_{2d}(N)$ and its $(N-2d)$-dimensional complement $U_{2d}(N)^\perp$. The Thom spaces (one point compactifications) give rise to the spectrum $\MT^{\theta_k}(2d)$ with structure maps
$$\Th(\theta_k^*(U_{2d}(N)^\perp))\wedge S^1 \rightarrow \Th(\theta_k^*(U_{2d}(N+1)^\perp)).$$
More precisely, the $N$'th space of $\MT^{\theta_k}(2d)$ is the Thom space of $\theta_k^*(U_{2d}(N)^\perp)$. The associated infinite loop space is by definition
\begin{equation} \label{eq:(3)}
\Omega^\infty \MT^{\theta_k}(2d) = \hocolim_N \Omega^N \Th(\theta_k^*(U_{2d}(N)^\perp)).
\end{equation}

Just as surfaces come in two varieties, orientable or non-orientable, so do $(d-1)$-connected $2d$-dimensional manifolds $E$ for $d>1$. The tangent bundle $TE$ is represented by a map $\tau\colon E\rightarrow BO(2d)$, and we may consider
$$\tau_*\colon \pi_d E \rightarrow \pi_d BO(2d) \cong \pi_d BO.$$
The target is a cyclic group. Let $f = f(E)$ be a generator. The case $f(E) = 0$ corresponds to orientable surfaces and $f(E) \ne 0$ to non-orientable surfaces when $d=1$. Since $\pi_dBO \cong \pi_d BO(d+1)$, $f$ gives rise to a $(d+1)$-dimensional vector bundle over $S^d$. Its associated sphere bundle is denoted by $K$. Notice that $K = S^d\times S^d$ if $f = 0$ and otherwise $K$ is a $2d$-dimensional ``Klein bottle''. Consider the string
$$B\Diff_D(E)\rightarrow B\Diff_D(E\#K) \rightarrow \ldots \rightarrow B\Diff_D(E\#gK)\rightarrow \ldots,$$
where $D\subset E$ is a $2d$-dimensional disk and $\Diff_D(E)$ etc denotes the diffeomorphisms that fix $D$ pointwise. Set
$$\theta_f = \left\{ \begin{array}{cc} \theta_{d+1}, & \mbox{if $f=0$} \\ \theta_d, & \mbox{if $f\ne 0$.} \end{array} \right.$$
The main result of \cite{GR-W} is
\begin{theorem}[Galatius, Randal-Williams] \label{thm:GRW}
For $d\ne 2$ and any $(d-1)$-connected $2d$-dimensional closed manifold $E$,
$$\colim_g \HH_*(B\Diff_D(E\#gK);\ZZ) \cong \HH_*(\Omega_0^\infty \MT^{\theta_f}(2d);\ZZ),$$
where the subscript in the target indicates the connected component of the constant loop.
\end{theorem}

For $d=1$, this is the generalized Mumford conjecture from \cite{MW}. The theorem raises the obvious question if there is a stability range for the colimit, depending on $g$. This is the case when $d=1$ by \cite{Bo,Wahl}. We introduce the notation
\begin{equation} \label{eq:(4)}
M_g = (S^d\times S^d)\# \ldots \# (S^d\times S^d),\quad \mbox{$g$ summands.}
\end{equation}

Our main result is the following
\begin{theorem} \label{thm:main}
For $d>2$,
$$\HH_k(B\Diff_D(M_g^{2d});\QQ) \rightarrow \HH_k(B\Diff_D(M_{g+1}^{2d});\QQ)$$
is an isomorphism provided $k< \min(d-2,\frac{1}{2}(g-5))$.
\end{theorem}

We believe that there is a similar stability result in the ``unoriented'' case, where $M_g$ is replaced with
$$N_g = K\#\ldots \# K,\quad \mbox{$g$ summands,}$$
and $K$ is a generalized Klein bottle, i.e., the sphere bundle of the vector bundle resprented by a non-trivial map $f\colon S^d\rightarrow BO(d+1)$.

\begin{remark} \label{rmk:1.3}
Theorem \ref{thm:main} implies a stability result for the spaces in Theorem \ref{thm:GRW} when $E$ is ``oriented''. This follows because for a given $E$ there exists an $F$ with $E\# F \# M_h \cong M_g$ for suitable $h$ and $g$.
\end{remark}

The rational cohomology of the right hand side of Theorem \ref{thm:GRW} is easily displayed. First recall the rational homotopy type of $BSO(2d)$ is
$$BSO(2d)_\QQ \simeq K(\QQ,2d) \times \prod_{\ell=1}^{d-1} K(\QQ,4\ell)$$
with the map given by the Euler class and the Pontryagin classes of the universal $2d$-dimensional vector bundle. This implies that
\begin{equation}
BO(2d)[d+1,\infty)_\QQ \simeq K(\QQ,2d)\times \prod_{\ell = \lceil \frac{d+1}{4} \rceil}^{d-1} K(\QQ,4\ell).
\end{equation}
The Thom isomorphism theorem calculates the rational cohomology of the spectrum $\MT^\theta(2d)$,
$$\HH^k(\MT^\theta(2d);\QQ) \cong \HH^{k-2d}(BO(2d)[d+1,\infty);\QQ).$$
Indeed, by definition
\begin{align*}
\HH^k(\MT^\theta(2d)) & = \lim \HH^{k+N}(\MT^\theta(2d)_N)) = \lim \HH^{k+N}(\Th(\theta_N^* U_{2d}(N)^\perp)) \\
& = \lim \HH^{k-2d}(\Grass_{2d}(\RR^N)) = \HH^{k-2d}(BO(2d)[d+1,\infty)),
\end{align*}
where the inverse limit is for $N\rightarrow \infty$. The rational cohomology of the connected component $\Omega_0^\infty \MT^\theta(2d)$ can be expressed as
\begin{equation} \label{1.5}
\HH^*(\Omega_0^\infty \MT^\theta(2d);\QQ) = \Lambda (\HH^{*>2d}(BO(2d)[d+1,\infty);\QQ)[-2d]).
\end{equation}
Here $\Lambda$ denotes the free graded commutative algebra of the stated graded vector space. The bracket $[-2d]$ is the notation for shifting down grading by $2d$. For example,
$$\HH^*(\Omega_0^\infty\MT^\theta(2);\QQ) = \Lambda(\spann_\QQ \set{x_i}{\deg x_i = 2i > 2}[-2])$$
is a polynomial algebra in generators of degrees 2, 4, 6 ..., and
$$\HH^*(\Omega_0^\infty \MT^\theta(6);\QQ) = \Lambda(\spann_\QQ \set{x_{ijk}}{\deg x_{ijk} = 4i+6j+8k>6}[-6])$$
because the rational cohomology algebra of $BO(6)[4,\infty) = B\Spin(6)$ is a polynomial algebra in the first two Pontryagin classes and the Euler class. In general, the rational cohomology of $BO(2d)[d+1,\infty)$ is concentrated in even dimensions and the algebra in \eqref{1.5} is a polynomial algebra.

The proof of Theorem \ref{thm:main} uses rational homotopy theory, the surgery exact sequence and Morlet's lemma of disjunction. Let $M$ be a $(d-1)$-connected, $2d$-dimensional closed manifold. By combining Quillen's and Sullivan's models for rational homotopy and Koszul duality, we derive an explicit formula for the rational homotopy groups of the space of homotopy self-equivalences $\aut(M)$ in the case when $d\geq 2$ and $n = \rank \HH^d(M) \geq 3$. In particular, we find that $\pi_k(\aut(M))\tensor \QQ = 0$ unless $k$ is divisible by $d-1$. Moreover, we prove that the natural map
$$\pi_0\aut(M)\rightarrow \Aut_{alg}(\HH^*(M;\ZZ))$$
has finite kernel and image of finite index. For the manifolds of \eqref{eq:(4)}, this implies that $\pi_0\aut(M)$ is commensurable with the symplectic group $\Sp_{2g}(\ZZ)$ or the orthogonal group $O_{2g}(\ZZ)$ depending on the parity of $d$.

The next step in the proof of Theorem \ref{thm:main} is to compare $\aut(M)$ to the group of block diffeomorphisms $\tDiff(M)$, but in a relative case. Let $N = M_g \setminus \interior D$, with $M_g$ the manifold displayed in \eqref{eq:(4)}. There is a fibration
\begin{equation} \label{eq:(8)}
\tAut_{\partial N}(N)/\tDiff_{\partial N}(N) \rightarrow B\tDiff_{\partial N}(N) \rightarrow B\tAut_{\partial N}(N),
\end{equation}
where $\tAut_{\partial N}(N)\simeq \aut_{\partial N}(N)$ denotes the topological monoid of block homotopy self-equivalences of $N$ that keep $\partial N$ pointwise fixed. We note that $\tDiff_{\partial N}(N) \cong \tDiff_D(M_g)$. The surgery exact sequence is used to show that
$$\pi_k(\tAut_{\partial N}(N)/\tDiff_{\partial N}(N))\tensor \QQ = \HH^d(M_g;\QQ)\tensor \pi_{k+d}(G/O)$$
and that
$$\HH_*(\tAut_{\partial N}(N)/\tDiff_{\partial N}(N);\QQ) = \Lambda(\pi_*\tensor \QQ).$$
Stability for group cohomology of $\Sp_{2g}(\ZZ)$ and $O_{2g}(\ZZ)$ with coefficients in standard modules leads to the analogue of Theorem \ref{thm:main} with $\Diff_D(M)$ replaced with $\tDiff_D(M)$. Finally, Morlet's lemma of disjunction completes the proof of Theorem \ref{thm:main}.

One can also use \eqref{eq:(8)} to describe the rational homotopy groups of $B\tDiff_D(M_g)$. Indeed, $\pi_*(\aut_{\partial N}(N))\tensor \QQ$ vanishes in degrees $0<*<d-1$. In combination with Morlet's lemma, this gives
$$\pi_{k-1}(\Diff_D(M_g))\tensor \QQ \cong \HH^d(M_g;\QQ)\tensor \pi_{k+d}(G/O)$$
for $1<k<d-1$. More generally, one could hope to describe the rational homotopy groups of $B\tDiff_D(M_g)$ completely. This would require knowledge of
$$\pi_{k+1}(B\tAut_{\partial N}(N))\tensor \QQ \rightarrow \pi_k(\tAut_{\partial N}(N)/\tDiff_{\partial N}(N))\tensor \QQ.$$
The source is concentrated in degrees $k\equiv 0 \,\,(d-1)$ and the target in degrees $k \equiv -d \,\,(4)$. We conjecture that
$$\HH_k(B\tDiff_D(M_g);\QQ) \rightarrow \HH_k(B\tDiff_D(M_{g+1});\QQ)$$
is an isomorphism for $k<\frac{g-5}{2}$, independent of the dimension of $M_g$ provided it is $\ne 4$. We hope to return to this in a planned successor to this paper.

Finally it is a pleasure to thank Oscar Randal-Williams and S\o{}ren Galatius for helpful conversations about the content of this paper.

\section{Rational homotopy theory of homotopy self-equivalences} \label{sec:2}
Let $M$ be a $(d-1)$-connected $2d$-dimensional closed manifold where $d\geq 2$. Let $D^{2d}\subset M$ be an embedded disk and let $N$ be the manifold with boundary $N = M\setminus \interior D^{2d}$. For a pair $A \subset X$ of topological spaces, let $\aut_A(X)$ denote the topological monoid of homotopy self-equivalences of $X$ that leaves $A$ fixed pointwise, with the compact-open topology. We calculate below the rational homotopy groups of $\aut_{\partial N}(N)$, $\aut_D(M)$ and $\aut(M)$. In particular, we find that in all three cases,
$$\pi_k(\aut)\tensor \QQ = 0,\quad \mbox{for $k\not\equiv 0$ mod $(d-1)$.}$$
Furthermore, we show that the evident homomorphism to the automorphism group of the cohomology algebra
$$h\colon \pi_0(\aut) \rightarrow \Aut_{alg} \HH^*(M;\ZZ)$$
has finite kernel and its image is a subgroup of finite index.

Our proof will depend on methods from rational homotopy theory. The rational homotopy type of a simply connected space $X$ with degreewise finite dimensional rational homology is modeled algebraically by either
\begin{itemize}
\item The Sullivan-de Rham commutative differential graded algebra (cdga) $\Omega(X)$ of polynomial differential forms on $X$ \cite{Sullivan}.
\item Quillen's differential graded Lie algebra (dgl) $\lambda(X)$ of normalized chains on the simplicial Lie algebra of primitives in the completed group algebra on the Kan loop group of a 1-reduced simplicial set model for $X$ \cite{Quillen}.
\end{itemize}

\subsection{Sullivan's rational homotopy theory}
Let $\Omega_\bullet$ denote the simplicial polynomial de Rham algebra; it is the simplicial cdga with $n$-simplices
$$\Omega_n = \frac{\QQ[t_0,\ldots,t_n]\tensor \Lambda(dt_0,\ldots,dt_n)}{(t_0+\ldots + t_n-1, dt_0+\ldots +dt_n)},\quad |t_i| = 0,\,|dt_i| = 1,$$
with standard face and degeneracy maps. The simplicial deRham algebra $\Omega_\bullet$ gives rise to a (contravariant) adjunction between simplicial sets and commutative differential graded algebras
$$\adjunction{\mathsf{sSet}}{\mathsf{CDGA}_\QQ^{op},}{\Omega}{\langle - \rangle}$$
$$\Omega(X) = \Hom_{sSet}(X,\Omega_\bullet),\quad \langle B \rangle = \Hom_{cdga}(B,\Omega_\bullet).$$
The cdga $\Omega(X)$ is the Sullivan-deRham algebra of polynomial differential forms on $X$, and $\langle B \rangle$ is the spatial realization of $B$. It is a fundamental result in rational homotopy theory that the adjunction induces an equivalence between the homotopy categories of nilpotent rational spaces of finite $\QQ$-type and minimal algebras of finite type, see \cite{BG,Sullivan}.

If $T$ is a topological space, then we define $\Omega(T) = \Omega(S_\bullet(T))$, where $S_\bullet(T)$ is the singular complex of $T$. We prefer to work in the category of simplicial sets and will use the word `space' to mean a simplicial set unless otherwise indicated. A cdga $A$ which is quasi-isomorphic\footnote{$A$ is \emph{quasi-isomorphic} to $B$ if there is a zig-zag of morphisms of differential graded algebras that induce isomorphisms in cohomology $A\stackrel{\sim}{\leftarrow} \ldots \stackrel{\sim}{\rightarrow} B$.} to $\Omega(X)$ is called a \emph{cdga model} for $X$. A \emph{Sullivan algebra} is a cdga of the form $A = (\Lambda V,d)$ where
\begin{itemize}
\item $\Lambda V$ denotes the free graded commutative algebra on a graded $\QQ$-vector space $V = V^1 \oplus V^2 \oplus \ldots$.
\item $V$ admits a filtration $0 = F_0V \subseteq F_1V \subseteq \ldots \subseteq \cup_p F_pV = V$ such that
$$d(F_pV) \subseteq \Lambda(F_{p-1}V).$$
\end{itemize}
A Sullivan algebra which is also a cdga model for $X$ is called a \emph{Sullivan model} for $X$. Sullivan algebras play the role of CW-complexes in the category of cdgas; they are \emph{cofibrant} in the sense that, up to homotopy, maps from a Sullivan algebra to the target of a quasi-isomorphism lift uniquely to the source.

\begin{theorem}[Sullivan's localization theorem {\cite{BG,Sullivan}}]
Let $X$ be a nilpotent connected space of finite $\QQ$-type and let $A\stackrel{\sim}{\rightarrow} \Omega(X)$ be a Sullivan model. Then the adjoint map $X\rightarrow \langle A \rangle$ is a $\QQ$-localization.
\end{theorem}

\subsection{Quillen's rational homotopy theory}
For a simplicial set $X$ with $X_0 = X_1 = *$, Quillen \cite{Quillen} considers the differential graded Lie algebra
$$\lambda(X) = N_*\mathscr{P}\widehat{\QQ}[G_\bullet X]$$
of normalized chains on the simplicial Lie algebra of primitives in the completed group algebra on the Kan loop group of $X$. The functor $\lambda\colon \mathsf{sSet}_1\rightarrow DGL$ induces an equivalence between the homotopy category of $1$-connected rational spaces and dgls. The homology of $\lambda(X)$ is isomorphic to the graded Lie algebra formed by the rational homotopy groups $\pi_*(\Omega X)\tensor \QQ$ with Samelson products:
$$\HH_*(\lambda(X)) \cong \mathscr{P}\HH_*(\Omega X;\QQ) \cong \pi_*(\Omega X)\tensor \QQ.$$
A dgl $L$ which is quasi-isomorphic to $\lambda(X)$ in the category of differential graded Lie algebras is called a \emph{dgl model} for $X$. A \emph{Quillen model} for $X$ is a dgl model of the form $L = (\LL(W),d)$ where $\LL(W)$ denotes the free graded Lie algebra on a graded vector space $W = W_1\oplus W_2 \oplus \ldots$. Quillen models are cofibrant in the category of dg Lie algebras.

Given a dgl $L$ of finite type, let $C_*(L)$ be the differential graded cocommutative coalgebra
$$C_*(L) = (\Lambda(L[1]),d = d_0 + d_1),\quad L[1]_k = L_{k-1},$$
where $d_0$ is induced from the differential in $L$ and
$$d_1(x_1\wedge \ldots \wedge x_k) = \sum_{i<j} \epsilon_{ij} [x_i,x_j] \wedge x_1 \wedge \ldots \widehat{x_i} \ldots \widehat{x_j} \ldots \wedge x_k.$$
Here $\epsilon_{ij}$ is the standard sign from permuting $x_i$ and $x_j$ to the front. The \emph{Chevalley-Eilenberg construction} $C^*(L)$ is the dual cdga,
$$C^*(L) = \Hom_\QQ(C_*(L),\QQ) \cong \Lambda(L[1]^\vee)$$
with $L[1]^\vee = \Hom_\QQ(L[1],\QQ)$.

In 1977 Baues and Lemaire \cite{BL} conjectured that the Chevalley-Eilenberg construction would provide a bridge from Quillen's to Sullivan's theory. This was proved by Majewski in 2000 \cite{Majewski}. The precise statement is the following.
\begin{theorem}[Baues-Lemaire conjecture \cite{BL,Majewski}]
Let $X$ be a simply connected space of finite $\QQ$-type. If $L$ is a finite type dgl model for $X$ then $C^*(L)$ is a Sullivan model for $X$.
\end{theorem}

\subsection{Rational homotopy theory of mapping spaces}
There are several approaches to the rational homotopy theory of mapping spaces, see for instance \cite{B2,BS,BFM,Haefliger}. We will describe the model of \cite{B2}, which uses the Maurer-Cartan simplicial set associated to a dg Lie algebra.

Let $\gl$ be a dgl whose underlying chain complex is not necessarily bounded. The set of \emph{Maurer-Cartan elements} $\MC(\gl)$ consists of all elements $\tau\in \gl$ of degree $-1$ that satisfy the equation
\begin{equation} \label{mc}
d(\tau) + \frac{1}{2}[\tau,\tau] = 0.
\end{equation}
Equivalently, $\tau$ is a Maurer-Cartan element if the map $d^\tau\colon \gl\rightarrow \gl$ defined by
$$d^\tau(x) = dx + [\tau,x]$$
satisfies $(d^\tau)^2 = 0$. In this case, $(\gl,d^\tau)$ is again a dgl.

If $A$ is a cdga and $\gl$ is a dgl, the tensor product chain complex $A\tensor \gl$ becomes a dgl with
Lie bracket
$$[x\tensor \alpha,y\tensor \beta] = (-1)^{|\alpha||y|}xy\tensor [\alpha,\beta],$$
and grading
$$(A\tensor \gl)_k = \bigoplus_n A^n\tensor \gl_{n+k}.$$

Following \cite{G}, we may form the simplicial dgl $\Omega_\bullet \tensor \gl$ and the simplicial set
$$\MC_\bullet(\gl) = \MC(\Omega_\bullet \tensor \gl).$$
Since $\tau \mapsto d\tau + \frac{1}{2}[\tau,\tau]$ is not a linear operator, $\MC_\bullet(\gl)$ is not a simplicial group, nevertheless it is fibrant, i.e. a Kan complex, by \cite{G}.

The path components of $\MC_\bullet(\gl)$ is the set of Maurer-Cartan elements of $\gl$ modulo homotopy equivalence. Let us make this explicit. A $1$-simplex $\gamma\in \mathsf{MC}_1(\gl) = \mathsf{MC}(\Lambda(t,dt)\tensor \gl)$ may be written as
$$\gamma = \alpha(t) + \beta(t) dt$$
where $\alpha(t) = \sum_i \alpha_i t^i$ and $\beta(t) =\sum_i \beta_i t^i$ for elements $\alpha_i \in \gl_{-1}$ and $\beta_i\in \gl_0$, and the Maurer-Cartan equation $d\gamma + \frac{1}{2}[\gamma,\gamma] = 0$ is equivalent to
\begin{align*}
d \alpha(t) + \frac{1}{2}[\alpha(t),\alpha(t)] & = 0 \\
d\beta(t) + [\alpha(t),\beta(t)]  & = \dot{\alpha}(t)
\end{align*}
where $d\alpha(t) = \sum_i d(\alpha_i)t^i$ and $\dot{\alpha}(t) = \sum_i i\alpha_i t^{i-1}$. Two Maurer-Cartan elements $\tau,\tau'\in \mathsf{MC}(\gl)$ are homotopy equivalent if there is a $\gamma$ as above such that $\gamma|_{t=0} = \tau$ and $\gamma|_{t=1} = \tau'$, or equivalently, $\alpha(0) = \tau$ and $\alpha(1) = \tau'$.

For a fixed basepoint $\tau\in\MC_\bullet(\gl)$, i.e., a Maurer-Cartan element of $\gl$, the higher homotopy groups $\pi_{n+1}(\MC_\bullet(\gl),\tau)$ can be calculated as follows:

\begin{proposition}
Let $\tau$ be a Maurer-Cartan element of $\gl$. For every $n\geq 0$ the map
\begin{equation} \label{eq:homotopy}
B_n^\tau\colon \HH_n(\gl,d^\tau) \stackrel{\cong}{\rightarrow} \pi_{n+1}(\MC_\bullet(\gl),\tau),
\end{equation}
$$[\alpha] \mapsto [ 1\tensor \tau + dt_0\wedge \ldots \wedge dt_n\tensor \alpha]$$
is an isomorphism of abelian groups for $n\geq 1$, and for $n = 0$ it identifies the group $\pi_1(\MC_\bullet(\gl),\tau)$ with the exponential of the Lie algebra $\HH_0(\gl,d^\tau)$.

\end{proposition}
We refer the reader to \cite{B2} for proofs of these facts.

Notice that if $L$ is a dgl that is concentrated in positive degrees, then $\MC_\bullet(L)$ is simply connected. In a sense, the construction $L\mapsto \MC_\bullet(L)$ is inverse to Quillen's functor $Y\mapsto \lambda(Y)$. Indeed, suppose that $L$ is a finite type dgl model for a simply connected space $Y$ of finite $\QQ$-type. For a bounded cdga $A$, restriction to generators gives a map
$$\Hom_{cdga}(C^*(L),A)\rightarrow \Hom_\QQ(L[1]^\vee,A)_0\cong (A\tensor L)_{-1}.$$
Since the underlying algebra of $C^*(L)$ is free this map is injective. The image can be identified with $\MC(A\tensor L)$. In particular, if we take $A = \Omega_\bullet$ this yields an isomorphism of simplicial sets
$$\langle C^*(L) \rangle \stackrel{\cong}{\rightarrow} \MC_\bullet(L).$$
By the affirmed Baues-Lemaire conjecture, $C^*(L)$ is a Sullivan model for $Y$, i.e., there is a quasi-isomorphism of cdgas $\phi\colon C^*(L)\stackrel{\sim}{\rightarrow} \Omega(Y)$. By Sullivan's localization theorem, the adjoint $Y \rightarrow \langle C^*(L) \rangle \cong \MC_\bullet(L)$ to $\phi$ is a $\QQ$-localization. Thus, if $L$ is a finite type dgl model for $\lambda(Y)$, then $\MC_\bullet(L)$ is a $\QQ$-localization of $Y$.

Next, let $X$ be a finite connected space. For a fixed map $f\colon X\rightarrow Y$ the $\QQ$-localization map $r\colon Y\rightarrow \MC_\bullet(L)$ induces a $\QQ$-localization
$$r_*\colon\Map(X,Y)_f\rightarrow \Map(X,\MC_\bullet(L))_{rf},$$
see e.g., \cite[Theorem II.3.11]{HMR}. There is a natural homotopy equivalence
\begin{equation} \label{eq:heq}
\varphi\colon\MC_\bullet(\Omega(X)\tensor L) \rightarrow \Map_{sSet}(X,\MC_\bullet(L))
\end{equation}
which is defined as follows: Firstly, there is a natural isomorphism
$$\mu\colon \MC(\Omega(X)\tensor L) \rightarrow \Hom_{sSet}(X,\MC_\bullet(L))$$
given by $\mu(\tau)(x) = x^*(\tau)$, where $x^*\colon \Omega(X)\tensor L\rightarrow \Omega_n\tensor L$ is the dgl morphism induced by a simplex $x\in X_n$. Secondly, on $k$-simplices $\varphi$ defined as the composite
$$\MC(\Omega(\Delta[k])\tensor \Omega(X)\tensor L) \rightarrow \MC(\Omega(\Delta[k]\times X)\tensor L) \stackrel{\mu}{\rightarrow} \Map_{sSet}(X,\MC_\bullet(L))_k$$
where $\pi$ is induced by the natural morphism $\Omega(X)\tensor \Omega(Y)\rightarrow \Omega(X\times Y)$. The proof that $\varphi$ is a homotopy equivalence can be found in \cite{B2}. Furthermore, the functor $\MC_\bullet(-\tensor L)\colon \mathsf{CDGA}_\QQ \rightarrow \mathsf{sSet}$ takes quasi-isomorphisms to homotopy equivalences \cite{B2}. Therefore, the mapping space $\Map(X,Y_\QQ)$ is homotopy equivalent to $\MC_\bullet(A\tensor L)$ where $A$ is any cdga model for $X$ and $L$ is any dgl model for $Y$. The homotopy groups may be calculated by \eqref{eq:homotopy}.

\begin{theorem} \label{thm:map}
Let $X$ be a connected finite space and let $Y$ be a simply connected space of finite $\QQ$-type. If $A$ is a cdga model for $X$ and $L$ is a finite type dgl model for $Y$ then there is a homotopy equivalence
$$\MC_\bullet(A\tensor L) \stackrel{\sim}{\rightarrow} \Map(X,Y_\QQ).$$
In particular, there is a bijection
\begin{equation} \label{bijection}
[X,Y_\QQ] \cong \pi_0\MC_\bullet(A\tensor L).
\end{equation}
Let $f\colon X\rightarrow Y$ be a fixed map, and let $\tau\in \MC(A\tensor L)$ be a Maurer-Cartan element whose component corresponds to the homotopy class of $rf\colon X\rightarrow Y_\QQ$ under the bijection \eqref{bijection}. Then for every $n\geq 0$ there is an isomorphism
$$\HH_n(A\tensor L, d^\tau)\stackrel{\cong}{\rightarrow} \pi_{n+1}(\Map(X,Y),f)\tensor \QQ.$$
For $n\geq 1$ this is an isomorphism of rational vector spaces and for $n=0$ it identifies the Malcev completion of the fundamental group $\pi_1(\Map(X,Y),f)$ with the exponential of the Lie algebra $\HH_0(A\tensor L,d^\tau)$.
\end{theorem}

This theorem is proved in \cite{B2}. See also \cite{BFM}. The constant map $X\rightarrow Y$ corresponds to the trivial Maurer-Cartan element $\tau = 0$. The finiteness assumption on $X$ can be relaxed if one considers coalgebra models instead of cdga models. Moreover, one can relax $L$ to a be an $L_\infty$-algebra rather than a strict dgl.

\subsection{Formal and coformal spaces and Koszul algebras}
Theorem \ref{thm:map} raises the question of how to find tractable algebraic models for the spaces involved. The cohomology of the Sullivan-de Rham algebra $\Omega(X)$ is isomorphic to the singular cohomology algebra $\HH^*(X;\QQ)$. The simplest possible cdga with this cohomology is the cohomology itself, viewed as a cdga with zero differential. If $\Omega(X)$ is indeed quasi-isomorphic to $\HH^*(X;\QQ)$, then $X$ is called \emph{formal}. Far from all spaces are formal, Massey operations in the cohomology being a first obstruction, but sometimes formality is forced by geometric constraints. A celebrated result due to Deligne, Griffiths, Morgan and Sullivan says that every simply connected compact K\"ahler manifold is formal \cite{DGMS}. Formality can also be deduced from connectivity and dimension constraints. As shown in \cite{Mi}, every $(d-1)$-connected manifold of dimension at most $4d-2$ is formal, for $d\geq 2$.

There is a parallel story for the Quillen model. The homology of Quillen's dgl $\lambda(X)$ is isomorphic to the graded Lie algebra formed by the rational homotopy groups $\pi_*(\Omega X)\tensor \QQ$ with Samelson products. A space $X$ is called \emph{coformal} if the homotopy Lie algebra $\pi_*(\Omega X)\tensor \QQ$, with zero differential, is a dgl model for $X$.

If we want to use Theorem \ref{thm:map} to model a mapping space $\Map(X,Y)$, the simplest case imaginable is when $X$ is formal and $Y$ is coformal; then one may choose $A = \HH^*(X;\QQ)$ and $L = \pi_*(\Omega Y)\tensor\QQ$ as models. We are particularly interested in self-maps of $X$, and so one is naturally led to ask when $X$ is simultaneously formal and coformal. The answer is given by Theorem \ref{thm:koszul space} below. The characterization involves the notion of Koszul algebras, so we first need to explain what this means.

Koszul algebras were introduced by Priddy \cite{Priddy}. Let $A$ be a graded commutative $\QQ$-algebra $A = A^0 \oplus A^1 \oplus \ldots$ which is connected in the sense that $A^0 \cong \QQ$. Let $V_A = A^+/A^+ \cdot A^+$ denote the space of indecomposables and let $R_A \subseteq \Lambda^2V_A$ be the kernel of the multiplication map $\Lambda^2 V_A \rightarrow A$ induced by some choice of splitting of the projection $A^+ \rightarrow V_A$. The algebra $A$ is called \emph{quadratic} if the induced surjective morphism of graded algebras $\Lambda V_A/(R_A) \rightarrow A$ is an isomorphism. In other words, $A$ is quadratic if it is generated by some elements $x_i$ modulo certain quadratic relations
\begin{equation} \label{eq:relations}
\sum_{i,j} c_{ij} x_ix_j = 0,\quad c_{ij}\in\QQ.
\end{equation}
If $A$ is quadratic then there is an additional grading on $A$ given by the wordlength in the generators $x_i$\footnote{The generators $x_i$ need not be of the same cohomological degree.} This induces an additional grading also on the Ext-groups $\Ext_A^*(\QQ,\QQ)$. A graded commutative connected algebra $A$ is a \emph{Koszul algebra} if it is quadratic and if $\Ext_A^{s,t}(\QQ,\QQ) = 0$ for $s \ne t$.

Let $L = L_1 \oplus L_2 \oplus \ldots$ be a graded Lie algebra over $\QQ$, let $V_L = L/[L,L]$ denote the space of indecomposables and let $R_L\subseteq \LL^2(V_L)$ be the kernel of the multiplication map $\LL^2(V_L) \rightarrow L$ induced by some choice of splitting of the projection  $L\rightarrow V_L$. There is an induced surjective morphism of graded Lie algebras $\LL(V_L)/(R_L) \rightarrow L$. If it is an isomorphism then $L$ is called \emph{quadratic}. In other words, $L$ is quadratic if it is generated by some elements $\alpha_i$ modulo certain quadratic relations
\begin{equation} \label{eq:lie_relations}
\sum_{i,j} \lambda_{ij}[\alpha_i,\alpha_j] = 0,\quad \lambda_{ij}\in \QQ.
\end{equation}
There is an additional grading on a quadratic Lie algebra given by bracket length in the generators $\alpha_i$. This in turn induces an additional grading on the cohomology $\Ext_{UL}^*(\QQ,\QQ)$. A graded Lie algebra $L$ is called \emph{Koszul} if it is quadratic and if $\Ext_{UL}^{s,t}(\QQ,\QQ) = 0$ for $s\ne t$.

If $A$ is a quadratic commutative algebra and $L$ is a quadratic Lie algebra then we say that \emph{$A$ is Koszul dual to $L$} if there is a non-degenerate pairing of degree $1$
$$\langle \, , \, \rangle \colon V_A\tensor V_L \rightarrow \QQ$$
such that $R_L^\perp = R_A$ under the induced pairing
$$\langle \, ,\, \rangle \colon \Lambda^2(V_A)\tensor\LL^2(V_L)  \rightarrow \QQ,$$
$$\langle xy,[\alpha,\beta]\rangle = (-1)^{|y||\alpha|+|x|+|\alpha|} \langle x,\alpha\rangle \langle y,\beta\rangle - (-1)^{|\alpha||\beta|+|y||\beta| + |x| + |\beta|} \langle x,\beta\rangle \langle y,\alpha\rangle.$$
In other words, $A$ and $L$ are Koszul dual if they have dual generators and orthogonal relations. We may without loss of generality assume that the coefficients in \eqref{eq:relations} and \eqref{eq:lie_relations} are symmetric in the sense that $c_{ij} = (-1)^{|x_i||x_j|}c_{ji}$ and $\lambda_{ij} = -(-1)^{|\alpha_i||\alpha_j|}\lambda_{ji}$. In effect, to say that $A$ and $L$ have orthogonal relations means that a relation \eqref{eq:lie_relations} holds in $L$ if and only if
$$\sum_{i,j} (-1)^{|x_j||\alpha_i|} c_{ij}\lambda_{ij} = 0$$
whenever the coefficients $c_{ij}$ represent a relation among the generators $x_i$ as in \eqref{eq:relations}.

Every quadratic algebra $A$ has a dual, often denoted $A^{!Lie}$: simply define $V_L := \Hom_\QQ(V_A,\QQ)[-1]$, with the standard evaluation pairing, define $R_L:=R_A^\perp$ and let $L=\LL(V_L)/(R_L)$. Clearly, the dual is unique up to isomorphism.

There is a natural pairing between the cohomology and homotopy groups of a space $X$ given by
\begin{equation} \label{eq:pair}
\langle\,,\,\rangle\colon \HH^n(X)\tensor \pi_n(X) \rightarrow \ZZ,\quad \langle x ,\, \alpha \rangle = \langle \alpha^*(x), \, [S^n]\rangle.
\end{equation}
If $x$ is decomposable with respect to the cup product or if $\alpha$ is decomposable with respect to the Whitehead product then $\langle x,\, \alpha \rangle = 0$. Therefore, \eqref{eq:pair} induces a pairing (of degree $+1$) between indecomposables
\begin{equation} \label{eq:pairing}
\langle\,,\,\rangle\colon V_A\tensor V_L \rightarrow \QQ,
\end{equation}
where $A = \HH^*(X;\QQ)$ and $L = \pi_*(\Omega X)\tensor \QQ$. Sometimes the pairing is non-degenerate, sometimes it is not. If it is non-degenerate, then it sometimes exhibits $\HH^*(X;\QQ)$ and $\pi_*(\Omega X)\tensor\QQ$ as Koszul dual, but often it does not. However, we have the following theorem which characterizes spaces that are both formal and coformal.

\begin{theorem}[{\cite{B1}}] \label{thm:koszul space}
Let $X$ be a simply connected space of finite $\QQ$-type. The following are equivalent:
\begin{enumerate}
\item $X$ is both formal and coformal.
\item $X$ is formal and $\HH^*(X;\QQ)$ is a Koszul algebra.
\item $X$ is coformal and $\pi_*(\Omega X)\tensor \QQ$ is a Koszul Lie algebra.
\end{enumerate}
Furthermore, in this situation the pairing between indecomposables in cohomology and homotopy \eqref{eq:pairing} is non-degenerate and exhibits $\HH^*(X;\QQ)$ and $\pi_*(\Omega X)\tensor \QQ$ as Koszul dual to one another.
\end{theorem}

Returning to the problem of finding rational models for the space $\aut(X)$, note that the component of the mapping space $\Map(X,X)$ that contains a fixed homotopy self-equivalence is equal to the same component of $\aut(X)$, since any map homotopic to a homotopy equivalence is itself a homotopy equivalence. Moreover, $\pi_1(\aut(X),1_X)$ is an abelian group as $\aut(X)$ is a monoid. By combining Theorem \ref{thm:map} and Theorem \ref{thm:koszul space} we obtain the following.

\begin{theorem} \label{thm:chain}
Let $X$ be a simply connected finite complex. If $X$ is formal and coformal then there is an isomorphism of rational vector spaces
\begin{equation} \label{eq:aut}
\pi_{k+1}(\aut(X),1_X) \tensor \QQ \cong \HH_k(\HH^*(X;\QQ)\tensor \pi_*(\Omega X),[\kappa,-]),\quad k\geq 0.
\end{equation}
Here $\kappa = \sum_i x_i\tensor \alpha_i$, where $x_1,\ldots,x_n$ is a basis for the indecomposables of $\HH^*(X;\QQ)$ and $\alpha_1,\ldots,\alpha_n$ is a dual basis for the indecomposables of $\pi_*(\Omega X)\tensor \QQ$ under the natural pairing between cohomology and homotopy.
\end{theorem}

%

\subsection{Highly connected manifolds}
Consider a $(d-1)$-connected $2d$-dimensional closed manifold $M$ where $d\geq 2$. By Poincar\'e duality the cohomology of $M$ is of the form
$$\HH^*(M) = \HH^0(M) \oplus \HH^d(M) \oplus \HH^{2d}(M).$$
Moreover $\HH^0(M) \cong \HH^{2d}(M) \cong \ZZ$ and $\HH^d(M)\cong \ZZ^n$ for some $n$. Let $x_1,\ldots,x_n$ be a basis for $\HH^d(M)$. The cohomology algebra structure is completely determined by the integer $n\times n$-matrix $(q_{ij})$ where
$$q_{ij} = \langle x_i \cup x_j, [M] \rangle.$$
Since the cup product is graded commutative $q_{ij} = (-1)^d q_{ji}$ and $q_{ii} = 0$ if $d$ is odd. By the Hurewicz theorem there are classes $\alpha_i\in \pi_d(M)$ such that $\langle x_i,\alpha_j\rangle = \delta_{ij}$ under the pairing $\HH^d(M)\tensor \pi_d(M)\rightarrow \ZZ$.

By removing the interior of an embedded disk $D^{2d}\subset M$ we obtain a manifold $N := M\setminus \interior D^{2d}$ with boundary $\partial N \cong S^{2d-1}$. By reinserting the disk we recover the manifold $M$: there is a pushout square
\begin{equation} \label{eq:pushout}
\xymatrix{\partial N \ar[r] \ar[d]^i &  D^{2d} \ar[d]  \\ N \ar[r]^j & M.}
\end{equation}
The manifold $N$ is homotopy equivalent to an $n$-fold wedge of $d$-dimensional spheres, $N \simeq \vee^n S^d$, and the inclusion $i\colon \partial N \rightarrow N$ is determined up to homotopy by the corresponding class $Q\in \pi_{2d-1}(\vee^n S^d)$. By Hilton's calculation \cite{Hilton} an element $Q\in\pi_{2d-1}(\vee^n S^d)$ can be written uniquely as
\begin{equation} \label{eq:hilton}
Q = \sum_{i<j} a_{ij} [\iota_i,\iota_j] + \sum_i \iota_i \gamma_i,
\end{equation}
where $\iota_i\in \pi_d(\vee^n S^d)$ is the homotopy class of the inclusion of the $i^{th}$ wedge summand $S^d\rightarrow \vee^n S^d$, the $a_{ij}$ are integers and $\gamma_i\in\pi_{2d-1}(S^d)$. The coefficients $a_{ij}$ and the Hopf invariant of $\gamma_i$ can be read off from the intersection matrix: $a_{ij} = q_{ij}$ and $H(\gamma_i) = q_{ii}$, provided $\iota_i$ maps to $\alpha_i$ under the inclusion $N\subset M$. Let $K\subseteq \pi_{2d-1}(S^d)$ be the kernel of the Hopf invariant homomorphism $H\colon \pi_{2d-1}(S^d)\rightarrow \ZZ$. Then $K$ is a finite group and $\pi_{2d-1}(S^d) \cong \ZZ \oplus K$ if $d$ is even and $\pi_{2d-1}(S^d) = K$ if $d$ is odd. If the Hopf invariant of $\gamma_i$ is even (which must be the case if $d\ne 2,4, 8$ by Adams' famous theorem), then $\beta_i : = \gamma_i - \frac{H(\gamma_i)}{2} [\iota_i,\iota_i]$ has Hopf invariant $0$ and we may rewrite $Q$ as
\begin{equation} \label{eq:attaching map}
Q = \frac{1}{2}\sum_{i,j}q_{ij}[\iota_i,\iota_j] + \sum_i \iota_i \circ \beta_i
\end{equation}
If we agree to interpret $\frac{1}{2}[\iota_i,\iota_i]$ as an element of Hopf invariant $1$ when relevant, then the above expression remains valid in all cases. The class $\iota_i\colon S^d\rightarrow N$ is homotopic to an embedding with normal bundle $\nu_i$ represented by $[\nu_i]\in \pi_d BSO(d)$. The $J$-homomorphism
$$J\colon \pi_d BSO(d) \rightarrow \pi_{2d-1}(S^d)$$
maps $[\nu_i]$ to the element $\gamma_i$ of \eqref{eq:hilton}. In the special case $M = M_g = \#^g (S^d\times S^d)$, \eqref{eq:hilton} reduces to
\begin{equation} \label{eq:red}
Q = [\iota_1,\iota_{g+1}] + \ldots + [\iota_g,\iota_{2g}]
\end{equation}
when $\iota_1,\ldots,\iota_g$ are the inclusions into the first factor of the $g$ summands $S^d\times S^d$ and $\iota_{g+1},\ldots,\iota_{2g}$ are the inclusions into the second factor. Indeed, the normal bundles $\nu_i$ are all trivial, so the elements $\gamma_i$ in \eqref{eq:hilton} vanish.

The rational homotopy Lie algebra of a wedge of spheres $\vee^n S^d$ is a free graded Lie algebra
$$\pi_*(\Omega (\vee^n S^d)) \tensor \QQ \cong \LL(\iota_1,\ldots,\iota_n)$$
where the class $\iota_i\colon S^{d-1}\rightarrow \Omega (\vee^n S^d)$ is represented by the adjoint of the inclusion of the $i^{th}$ wedge summand \cite{Hilton}.

\begin{proposition} \label{prop:M}
Let $M$ be a $(d-1)$-connected $2d$-dimensional closed manifold with $d\geq 2$ such that $n = \rank \HH^d(M)\geq 2$. Then $M$ is formal and coformal and the rational homotopy Lie algebra of $M$ is given by
\begin{equation} \label{eq:liealg}
\pi_*(\Omega M) \tensor \QQ \cong \LL(\alpha_1,\ldots,\alpha_n)/(Q),\quad Q = \frac{1}{2}\sum_{i,j} q_{ij} [\alpha_i,\alpha_j],
\end{equation}
where $\alpha_i$ are classes of degree $d-1$ and $(q_{ij})$ is the cup product matrix of $M$.
\end{proposition}

\begin{proof}
That $M$ is formal and coformal follows from \cite[Proposition 4.4]{NM}. Hence, by Theorem \ref{thm:koszul space} the homotopy Lie algebra $L = \pi_*(\Omega M)\tensor \QQ$ is Koszul dual to the cohomology algebra $A=\HH^*(M;\QQ)$. This means that it is generated by the classes $\alpha_i\in \pi_d(M) \cong \pi_{d-1}(\Omega M)$ dual to $x_i\in \HH^d(M)$ modulo the orthogonal relations $R_L = R_A^\perp$. A relation
$$\sum_{i,j} c_{ij} x_i\cup x_j = 0,\quad c_{ij}\in\QQ,$$
holds in $\HH^*(M;\QQ)$ if and only if
$$\sum_{i,j} c_{ij} q_{ij} = \langle \sum_{i,j} c_{ij} x_i\cup x_j, [M] \rangle = 0.$$
Therefore, $R_A^\perp$ is one-dimensional and spanned by the single relation
$$\sum_{i,j} q_{ij} [\alpha_i,\alpha_j] = 0.$$
\end{proof}

There is a formula for the dimension $\eta_r$ of the weight $r$ component of the free graded Lie algebra  $\LL(\alpha_1,\ldots,\alpha_n)$ which is due Witt:
\begin{equation} \label{eq:witt}
\eta_r = \frac{1}{r}\sum_{\ell|r} \sigma \mu(\ell)n^{r/\ell}.
\end{equation}
Here $\sigma$ is the sign $\sigma = (-1)^{(d-1)(r+r/\ell)}$ and $\mu$ is the M\"obius function from number theory, i.e., $\mu(1) = 1$, $\mu(p_1\ldots p_m) = (-1)^m$ if $p_1,\ldots,p_m$ are distinct primes, and $\mu(\ell) = 0$ unless $\ell$ is square-free. There is a similar but more complicated formula for the dimensions of the Lie algebra $\LL(\alpha_1,\ldots,\alpha_n)/(Q)$.

\begin{proposition} \label{prop:epsilon}
With $M$ as in Proposition \ref{prop:M} the dimensions
$$\epsilon_r = \dim \pi_{r(d-1)}(\Omega M)\tensor \QQ$$
depend only on $n$ and the parity of $d$. They are given by the formula
\begin{equation} \label{eq:epsilon}
\epsilon_r = \sum_{\ell|r} \sigma \frac{\mu(\ell)}{\ell} \sum_{p+2q = r/\ell} \frac{(-1)^q n^p \binom{p+q}{p}}{p+q},
\end{equation}
where $\sigma = (-1)^{(d-1)(r+r/\ell)}$.
\end{proposition}

\begin{proof}
If $A$ is a graded commutative Koszul algebra with Koszul dual Lie algebra $L$, then there are certain numerical relations between the dimensions of the weight graded components of $A$ and those of the universal enveloping algebra $UL$, which express themselves as an equality of formal power series
\begin{equation} \label{eq:froberg}
P_A(-z) P_{UL}(z) = 1.
\end{equation}
Here
$$P_A(z) = \sum_{s\geq 0} \dim_\QQ A(s) z^s,$$
where $A(s)$ denotes the component of weight $s$. In our case, where $A = \HH^*(M;\QQ)$ and $L = \pi_*(\Omega M)\tensor \QQ$, we have that $P_A(z) = 1 + nz +z^2$, so \eqref{eq:froberg} implies that
\begin{equation} \label{eq:series}
P_{UL}(z) = \frac{1}{1-nz+z^2}.
\end{equation}
On the other hand, by the Poincar\'e-Birkhoff-Witt theorem there is an isomorphism of graded vector spaces between the universal enveloping algebra $UL$ and the free graded commutative algebra $\Lambda L$. Since the weight $r$ component $L(r)$ is concentrated in homological degree $r(d-1)$, this implies that
\begin{equation} \label{eq:PBW}
P_{UL}(z) = \left\{ \begin{array}{ll} (1-z)^{-\epsilon_1}(1-z^2)^{-\epsilon_2}(1-z^3)^{-\epsilon_3}\ldots, & \mbox{if $d$ is odd,} \\ (1+z)^{\epsilon_1}(1-z^2)^{-\epsilon_2}(1+z^3)^{\epsilon_3}(1-z^4)^{-\epsilon_4}\ldots, & \mbox{if $d$ is even.} \end{array} \right.
\end{equation}
We will deal with the case $d$ odd and leave the even case to the reader. If we take logarithms of \eqref{eq:series} and \eqref{eq:PBW} and expand in Taylor series we get, respectively,
$$\sum_{m \geq 1} \frac{(nz - z^2)^m}{m} = \sum_{\substack{p,q \\ p+q\geq 1}} \frac{(-1)^qn^p \binom{p+q}{p}}{p+q} z^{p+2q}$$
and
$$\sum_{k,\ell\geq 1} \epsilon_\ell\frac{z^{k\ell}}{k}.$$
When identifying $z^r$-coefficients we obtain
$$\sum_{p+2q = r} \frac{(-1)^qn^p \binom{p+q}{p}}{p+q} = \sum_{\ell | r} \epsilon_\ell\frac{\ell}{r},$$
and by applying the M\"obius inversion formula we arrive at the formula in \eqref{eq:epsilon}.
\end{proof}

The condition $n\geq 2$ in Proposition \ref{prop:M} is necessary as shown by the following example: The manifold $\CP{2}$ is formal, e.g., because it is K\"ahler, but it is not coformal because the cohomology $\HH^*(\CP{2};\QQ) = \QQ[x]/(x^3)$ is not a quadratic algebra, let alone a Koszul algebra. The rational homotopy Lie algebra $\pi_*(\Omega \CP{2})\tensor \QQ$ is abelian, with two indecomposable classes  $\alpha$ and $\beta$ in degrees $1$ and $4$, and does not have the form \eqref{eq:liealg}.

For $M_g = \#^g (S^d\times S^d)$ we have $n = 2g$, and the intersection matrix is
$$\left( \begin{array}{cc} 0 & I \\ (-1)^d I & 0 \end{array} \right),$$
where $I$ denotes the identity $g\times g$-matrix. Thus, $\pi_*(\Omega M_g)\tensor \QQ$ is generated by classes $\alpha_1,\ldots,\alpha_{2g}$ of degree $d-1$ modulo the relation
$$[\alpha_1,\alpha_{g+1}] + \ldots + [\alpha_g,\alpha_{2g}] = 0.$$

\subsection{Homotopy self-equivalences of highly connected manifolds}
For a pair of topological spaces $A\subset X$, let $\aut_A(X)$ denote the topological monoid of homotopy self-equivalences of $X$ that fix $A$ pointwise. We will now analyze the spaces of homotopy self-equivalences $\aut_{\partial N}(N)$, $\aut_*(M)$ and $\aut(M)$ where $M$ is a highly connected closed manifold and $N = M \setminus D$ where $D\subset M$ is an embedded disk of codimension $0$.

There is a map $e\colon \aut_{\partial N}(N) \rightarrow \aut_D(M)$ which extends a self-map of $N$ that fixes the boundary by the identity on the disk $D\subset M$. For diffeomorphisms or homeomorphisms the corresponding map is an isomorphism, the inverse being given by restriction to $N$, but for homotopy self-equivalences the map $e$ is not even a rational homotopy equivalence.

\begin{theorem} \label{thm:aut}
If $n = \rank \HH^d(M)\geq 3$ then the homotopy groups
$$\pi_k(\aut_{\partial N}(N),1_N), \quad \pi_k(\aut_*(M),1_M), \quad \pi_k(\aut(M),1_M),$$
are finite unless $k = r(d-1)$. For $r\geq 1$, the ranks are given by, respectively,
$$n \eta_{r+1} - \eta_{r+2}, \quad n\epsilon_{r+1} - \epsilon_{r+2}, \quad n\epsilon_{r+1} - \epsilon_{r+2} - \epsilon_r$$
where $\epsilon_r$ and $\eta_r$ are given by \eqref{eq:witt} and \eqref{eq:epsilon}. In particular, the ranks depend only on $n$ and the parity of $d$.
\end{theorem}

\begin{proof}
Since $M$ is formal and coformal we may use Theorem \ref{thm:chain} to calculate the rational homotopy groups of $\aut(M)$. Let $A=\HH^*(M;\QQ)$ and $L = \pi_*(\Omega M)\tensor \QQ$. Since $A(s) = 0$ for $s>2$, the chain complex $(A\tensor L,[\kappa,-])$ splits as a direct sum of chain complexes
$$\xymatrix{A(0)\tensor L(r) \ar[r]^-{[\kappa,-]} & A(1)\tensor L(r+1) \ar[r]^-{[\kappa,-]} & A(2)\tensor L(r+2).}$$
By inspection, this chain complex is isomorphic to the following:
\begin{equation} \label{eq:short}
\xymatrix{L(r) \ar[r]^-{\partial_1} & L(r+1)^n[-d] \ar[r]^-{\partial_0} & L(r+2)[-2d],}
\end{equation}
\begin{equation*}
\partial_1(\xi) = ([\alpha_1,\xi],\ldots,[\alpha_n,\xi]), \quad \partial_0(\zeta_1,\ldots,\zeta_n) = \sum_{i,j} q_{ij}[\alpha_i,\zeta_j].
\end{equation*}
Since the intersection form is non-degenerate, the matrix $(q_{ij})$ is invertible, and therefore $\partial_0$ is surjective. Since $\alpha_1,\ldots,\alpha_n$ generate the graded Lie algebra $L$, the kernel of $\partial_1$ may be identified with the center $Z(L)$. We will argue that the center is trivial if $n\geq 3$. To this end, we invoke \cite[Proposition 2]{Bog} which says that a graded Lie algebra of finite global dimension has non-trivial center only if the Euler characteristic $\chi(L)$ is zero, where
$$\chi(L) = \sum_i (-1)^i \dim_\QQ \Ext_{UL}^i(\QQ,\QQ),$$
and $UL$ denotes the universal enveloping algebra of $L$. In the situation at hand, $\Ext_{UL}^i(\QQ,\QQ) \cong A(i) \cong \HH^{id}(M;\QQ)$, so it follows that $L$ has global dimension $2$ and that $\chi(L) = 2-n$, whence $L$ must have trivial center whenever $n\ne 2$. In particular, the kernel of $\partial_1$ is trivial for $n\geq 3$. Thus, we can only have non-vanishing homology at the middle term of \eqref{eq:short}. It follows that the dimension of the middle homology group is
$$n\epsilon_{r+1} - \epsilon_r -\epsilon_{r+2},$$
where $\epsilon_r$ is the dimension of the component $L(r)$. The middle term is situated in homological degree $(r+1)(d-1)-d = r(d-1)-1$, so it follows from \eqref{eq:aut} that
$$\dim_\QQ \pi_{r(d-1)}(\aut(M),1_M)\tensor \QQ  = n\epsilon_{r+1} - \epsilon_r -\epsilon_{r+2},\quad r\geq 1.$$

The rational homotopy groups of $\aut_*(M)$ can be calculated as follows. The augmentation map $\epsilon \colon A\rightarrow \QQ$ is a model for the inclusion of the base-point $* \rightarrow M$. It follows from Theorem \ref{thm:map} and Theorem \ref{thm:chain} that the evaluation fibration
$$\aut_*(M)_{1} \rightarrow \aut(M)_{1} \rightarrow M$$
is modeled by the short exact sequence of differential graded Lie algebras
$$0\rightarrow (\overline{A}\tensor L,[\kappa,-]) \rightarrow (A\tensor L,[\kappa,-]) \stackrel{\epsilon \tensor 1}{\rightarrow} (L,0) \rightarrow 0$$
where $\overline{A} = \ker{\epsilon}$ denotes the augmentation ideal of $A$. The chain complex $(\overline{A}\tensor L,[\kappa,-])$ is the direct sum of
$$\xymatrix{A(1)\tensor L(r+1) \ar[r]^-{[\kappa,-]} & A(2)\tensor L(r+2)}$$
over all $r$ and the calculation above shows that the homology is concentrated in degree $r(d-1)-1$ and is of dimension $n\epsilon_{r+1} - \epsilon_{r+2}$.

Finally, to determine the rational homotopy groups of $\aut_{\partial N}(N)$ we consider the fibration
\begin{equation} \label{eq:fibration}
\aut_{\partial N}(N) \rightarrow \aut_*(N) \stackrel{i^*}{\rightarrow} \Map_*(\partial N,N).
\end{equation}
Since $N\simeq \vee^n S^d$ is formal and coformal, we may use Theorem \ref{thm:chain} to calculate the rational homotopy groups of $\aut_*(N)$. We find that
$$\pi_{*+1}(\aut_*(N),1_N)\tensor \QQ \cong \rHH(N;\QQ) \tensor \pi_*(\Omega N)\cong \LL^n[-d]$$
where $\LL = \pi_*(\Omega N)\tensor \QQ \cong \LL(\iota_1,\ldots,\iota_n)$ is the free graded Lie algebra on generators $\iota_i$ of degree $d-1$. The rational homotopy groups of $\Map_*(\partial N,N)$ can be calculated using Theorem \ref{thm:map}:
$$\pi_{*+1}(\Map_*(\partial N,N),i)\tensor \QQ \cong \rHH(S^{2d-1};\QQ)\tensor \pi_*(\Omega N)\cong \LL[-2d+1].$$
A calculation then shows that we may identify
$$i^* \colon \pi_{*+1}(\aut_*(N),1_N)\tensor \QQ \rightarrow \pi_{*+1}(\Map_*(\partial N,N),i)\tensor \QQ$$
with
\begin{equation} \label{eq:map}
\LL^n[-d] \rightarrow \LL[-2d+1]
\end{equation}
$$(\zeta_1,\ldots,\zeta_n)\mapsto \sum_{i,j} q_{ij}[\zeta_i,\iota_j].$$
Since $(q_{ij})$ is non-degenerate this map is surjective. Hence the rational homotopy sequence associated to the fibration \eqref{eq:fibration} splits and $\pi_{*+1}(\aut_{\partial N}(N),1_N)\tensor \QQ$ is isomorphic to the kernel of the map \eqref{eq:map}. A dimension counting argument finishes the proof.
\end{proof}

The condition $n\geq 3$ in Theorem \ref{thm:aut} is necessary to ensure that the Lie algebra $\pi_*(\Omega M)\tensor \QQ$ has trivial center. For $S^d\times S^d$ the Lie algebra is $\LL(\alpha_1,\alpha_2)/([\alpha_1,\alpha_2])$ which is abelian if $d$ is odd and has center spanned by $[\alpha_1,\alpha_1]$ and $[\alpha_2,\alpha_2]$ if $d$ is even.

\subsection{Groups of components}
We will now turn to the determination of the groups of components of the spaces $\aut_{\partial N}(N)$, $\aut_*(M)$ and $\aut(M)$ up to commensurability. First of all $\pi_0\aut_*(M) \cong\pi_0\aut(M)$ since $M$ is simply connected. The automorphism group of the cohomology algebra $\Aut_{alg}(\HH^*(M))$ is isomorphic to the automorphism group of the intersection form
$$\Aut(\ZZ^n,q) = \set{\lambda\in \GL_n(\ZZ)}{\lambda^t q \lambda = q}.$$
Recall that the attaching map $Q$ of the top cell in $M$ has the form
$$Q = \frac{1}{2}\sum_{i,j}q_{ij}[\iota_i,\iota_j] + \sum_i \iota_i \circ \beta_i$$
where $\beta_i\in \pi_{2d-1}(\vee^n S^d)$ have Hopf invariant $0$. Consider the subgroup
$$\Aut(\ZZ^n,q,\beta) = \set{\lambda\in \GL_n(\ZZ)}{\lambda^t q \lambda = q,\quad \lambda^t \beta = \beta} \subseteq \Aut(\ZZ^n,q).$$
This subgroup has finite index because it is an isotropy subgroup of the action of $\Aut(\ZZ^n,q)$ on the finite abelian group $K^n$, where $K = \ker(H\colon \pi_{2d-1}(S^d)\rightarrow \ZZ)$.

\begin{theorem} \label{thm:arithmetic}
There are group homomorphisms
$$\pi_0\aut_{\partial N}(N) \stackrel{e}{\rightarrow} \pi_0\aut_*^+(M) \rightarrow \Aut(\ZZ^n,q,\beta)$$
that are surjective and have finite kernels. In particular, all groups in the above sequence are commensurable with $\Aut(\ZZ^n,q) \cong \Aut_{alg}(\HH^*(M))$.
\end{theorem}

\begin{proof}
Consider the map $e\colon \Map_{\partial N}(N,N)\rightarrow \Map_*(M,M)$ which extends a self-map of $N$ by the identity on $D$. By construction $e(f)$ restricts to $f$ on $N$ so we get a commutative diagram of based spaces
\begin{equation} \label{eq:fiberseq}
\xymatrix{\Map_{\partial N}(N,N) \ar@{^{(}->}[r] \ar[d]^-e & \Map_*(N,N) \ar[r]^-{i^*} \ar[d]^{j_*} & \Map_*(\partial N,N) \ar[d]^-{j_*} \\ \Map_*(M,M) \ar[r]^-{j^*} & \Map_*(N,M) \ar[r]^-{i^*} & \Map_*(\partial N,M)}
\end{equation}
where the base-point in $\Map_{\partial N}(N,N)$ is the identity map. The upper row is a homotopy fiber sequence because $i\colon \partial N \rightarrow N$ is a cofibration and $\Map_{\partial N}(N,N)$ is the fiber over $i\in \Map_*(\partial N,N)$. The bottom row is a homotopy fiber sequence because it is induced by the homotopy cofiber sequence $\partial N \stackrel{i}{\rightarrow} N \stackrel{j}{\rightarrow} M$. The diagram \eqref{eq:fiberseq} gives rise to a commutative diagram with exact rows
\begin{equation} \label{eq:diag}
\xymatrix{
\pi_1(\Map_*(\partial N,N),i) \ar[r]^-\partial \ar[d]^-{j_*} & [N,N]_{\partial N} \ar[r] \ar[d]^-e & [N,N]_* \ar[r]^-{i^*} \ar[d]^-{j_*} & [\partial N,N]_* \ar[d]^-{j_*} \\
\pi_1(\Map_*(\partial N,M),ji) \ar[r]^-\partial & [M,M]_* \ar[r]^-{j^*} & [N,M]_* \ar[r]^-{i^*} & [\partial N,M]_*.
}
\end{equation}

The manifold $N$ is homotopy equivalent to $\vee^n S^d$, so the monoid $[N,N]_*$ may be identified with the monoid $M_{n\times n}(\ZZ)$ of $n\times n$ integer matrices. Indeed, a self-map $f$ of $\vee^n S^d$ gives an endomorphism $f_*$ of the abelian group $\pi_d(\vee^n S^d)$. The group $\pi_d(\vee^n S^d)$ is free abelian on $\iota_1,\ldots,\iota_n$, where $\iota_i$ is the homotopy class of the inclusion of the $i^{th}$ wedge summand $S^d\rightarrow \vee^n S^d$, so $f$ determines an $n\times n$ integer matrix $(\lambda_{ij})$ where
$$f_*(\iota_i) = \sum_{i,j} \lambda_{ij}\iota_j \in \pi_d(\vee^n S^d).$$
The map $[f] \mapsto (\lambda_{ij})$ gives an isomorphism of monoids $[\vee^n S^d,\vee^n S^d]_* \cong M_{n\times n}(\ZZ)$.

The boundary $\partial N$ is diffeomorphic to $S^{2d-1}$, so we may identify $[\partial N,N]_* \cong \pi_{2d-1}(\vee^n S^d)$. The map $\pi_0(i^*) \colon [N,N]_*\rightarrow [\partial N,N]_*$ is given by $[f]\mapsto i^*[f] = [f\circ i] = f_*[i]$, so by exactness of \eqref{eq:diag} the image of $[N,N]_{\partial N}$ in $[N,N]_*$ is the submonoid of all $[f]\in [N,N]_*$ such that $f_*[i] = [i]$. The inclusion $i\colon \partial N \rightarrow N$ is equivalent to the attaching map $Q\colon S^{2d-1} \rightarrow \vee^n S^d$ for the top cell of $M$ given by \eqref{eq:attaching map}. Now, for a map $f\colon \vee^n S^d\rightarrow \vee^n S^d$, we have
\begin{align*}
f_*(Q) & = \frac{1}{2}\sum_{i,j} q_{ij}[f_*(\iota_i),f_*(\iota_j)] + \sum_i f_*(\iota_i) \circ \beta_i \\
& = \frac{1}{2}\sum_{i,j,k,\ell} q_{ij}[\lambda_{ik} \iota_k,\lambda_{i\ell}\iota_{\ell}] + \sum_{i,j} \lambda_{ij}\iota_j \circ \beta_i \\
& = \frac{1}{2}\sum_{k,\ell} (\lambda^t q \lambda)_{k\ell} [\iota_k,\iota_{\ell}] + \sum_{j} \iota_j \circ (\lambda^t \beta)_j.
\end{align*}

Since the expression is unique, the above calculation shows that $f_*(Q) = Q$ if and only if
\begin{equation} \label{eq:condition}
\lambda^t q \lambda = q,\quad \lambda^t \beta = \beta.
\end{equation}

Therefore, the image of $[N,N]_{\partial N}$ in $[N,N]_*\cong M_{n\times n}(\ZZ)$ may be identified with the submonoid of all integer $n\times n$-matrices $\lambda$ such that \eqref{eq:condition} holds.

The map $j_*\colon [N,N]_* \rightarrow [N,M]_*$ is an isomorphism. The map $j_*\colon [\partial N,N]_* \rightarrow [\partial N,M]_*$ may be identified with the projection $\pi_{2d-1}(\vee^n S^d) \rightarrow \pi_{2d-1}(\vee^n S^d)/(Q)$. For this reason, the kernel of $i^*\colon [N,M]_* \rightarrow [\partial N,M]_*$ may be identified with the monoid of based homotopy classes of maps $f\colon \vee^n S^d\rightarrow \vee^n S^d$ such that $f_*(Q) = \ell Q \in \pi_{2d-1}(\vee^n S^d)$ for some $\ell \in \ZZ$. The condition $f_*(Q) = \ell Q$ is equivalent to
\begin{equation} \label{eq:condition2}
\lambda^t q \lambda = \ell q, \quad \lambda^t \beta = \ell \beta.
\end{equation}
Thus, the image of $[M,M]_*$ in $[N,M]_*$ is the submonoid of all $\lambda\in M_{n\times n}(\ZZ)$ such that \eqref{eq:condition2} is fulfilled. If $\lambda$ is invertible then necessarily $\ell = \pm 1$ and $\lambda$ comes from an orientation preserving homotopy self-equivalence of $M$ precisely when $\ell = 1$.

Passing to the submonoids of homotopy self-equivalences, the above shows that there is a commutative diagram of group extensions
$$
\xymatrix{1 \ar[r] & K \ar[r] \ar[d]^-{j_*} & \pi_0\aut_{\partial N}(N) \ar[r] \ar[d]^-{\pi_0(e)} & \Aut(\ZZ^n,q,\beta) \ar[r] \ar@{=}[d] & 1 \\
1 \ar[r] & L \ar[r] & \pi_0\aut_*^+(M) \ar[r] & \Aut(\ZZ^n,q,\beta) \ar[r] & 1}
$$
where the groups $K=\pi_1(\Map_*(\partial N,N),i)$ and $L=\pi_1(\Map_*(\partial N,M),ji)$ are finite. The map $j_*\colon K \rightarrow L$ is surjective. Indeed, let $F$ be the homotopy fiber of $j\colon N\rightarrow M$. The exact sequence derived from the homotopy fiber sequence $\Map_*(\partial N,F) \rightarrow \Map_*(\partial N,N) \rightarrow \Map_*(\partial N,M)$ looks like
$$\cdots \rightarrow K \stackrel{j_*}{\rightarrow} L \stackrel{\partial}{\rightarrow} \pi_{2d-1}(F) \stackrel{g}{\rightarrow} \pi_{2d-1}(N) \rightarrow \pi_{2d-1}(M).$$
The map $g$ is injective since it can be identified with $\ZZ \cong \pi_{2d-1}(F) \rightarrow \pi_{2d-1}(N)$ sending a generator to $[i]$. This implies that $j_*$ is surjective. By the five lemma, it follows that $\pi_0(e)$ is surjective.
\end{proof}

\section{Surgery theory and block diffeomorphisms}
In this section $M = M_g = \#^g(S^d\times S^d)$ and $N = M \setminus \interior D$ with $\partial N = S^{2d-1}$. The main objective of the section is the study of the rational homology of the pair
$$(\tAut_{\partial N}(N),\tDiff_{\partial N}(N)).$$
We remember that $\tAut_{\partial N}(N)$ is the topological monoid of block homotopy self-equivalences of $N$ that fix $\partial N$ pointwise, and that $\tDiff_{\partial N}(N)$ is the subgroup of block diffeomorphisms. Our tools are two homotopy fibrations\footnote{A homotopy fibration $F\rightarrow E\stackrel{\pi}{\rightarrow} B$ is a map $\pi$ whose homotopy fiber at $b\in B$ is homotopy equivalent to $F$.}:
\begin{equation} \label{eq:fibration1}
\tAut_{\partial N}(N)/\tDiff_{\partial N}(N) \rightarrow B\tDiff_{\partial N}(N) \rightarrow B\tAut_{\partial N}(N),
\end{equation}
\begin{equation} \label{eq:fibration2}
\tAut_{\partial N}(N)/\tDiff_{\partial N}(N) \stackrel{\eta}{\rightarrow} \Map_*(M,G/O) \stackrel{\lambda}{\rightarrow} \Lsp(M).
\end{equation}
The first fibration defines the ``homogeneous'' space in the fiber. The second fibration from \cite{Q,W} is F. Quinn's reformulation of the main theorem of surgery theory: its homotopy exact sequence is the surgery exact sequence of Sullivan and Wall in positive degrees. In combination with the results of \S \ref{sec:2}, we then derive a stability theorem for $B\tDiff_D(M_g)$, namely that the Harer type stabilization map
$$B\tDiff_D(M_g^{2d}) \rightarrow B\tDiff_D(M_{g+1}^{2d}),\quad d\ne 2,$$
induces an isomorphism in rational homology in a range of dimensions depending on $g$ and $d$. As $g$ and $d$ tend to infinity the range tends to infinity.

\subsection{The surgery exact sequence} \label{sec:surgery}
For an oriented compact manifold $X$ with boundary $\partial X$, the structure set $\Ss^{G/O}(X,\partial X)$ consists of pairs $(M,f)$ of a smooth manifold $M$ and a simple homotopy equivalence $f\colon (M,\partial M)\rightarrow (X,\partial X)$ with $\partial f \colon \partial M\rightarrow \partial X$ a diffeomorphism. Two pairs $(M_0,f_0)$ and $(M_1,f_1)$ define the same element of $\Ss^{G/O}(X,\partial X)$ if there is a diffeomorphism $\varphi\colon M_0\rightarrow M_1$ with $f_1\circ \varphi$ homotopic to $f_0$ rel. boundary.

The surgery exact sequence is a method to enumerate the structure set. We give a short review of it, referring to \cite{B,W} for full details.

A normal map with target $X$ is a pair $(f,\hat{f})$ consisting of a degree one map of manifolds
$$f\colon (M,\partial M) \rightarrow (X,\partial X),$$
with $\partial f$ a diffeomorphism, together with a fiberwise isomorphism of vector bundles
$$\hat{f} \colon \nu_M\rightarrow \zeta$$
over $f$. Here $\nu_M$ is the normal bundle of an embedding of $(M,\partial M)$ into a disk $(D^N,S^{N-1})$ with $N$ large compared to the dimension of $M$, and $\zeta$ is a vector bundle over $X$.

A normal cobordism between two normal maps $(f_0,\hat{f}_0)$ and $(f_1,\hat{f}_1)$ with target $X$ is a normal map $(F,\hat{F})$ where
$$F\colon W^{n+1}\rightarrow X\times I$$
is a degree one map from a relative cobordism between $(M_0,f_0)$ and $(M_1,f_1)$:
\begin{enumerate}
\item[(i)] $\partial W = M_0^n\cup V^n \cup M_1^n,\quad \partial V = \partial M_0\sqcup \partial M_1$.
\item[(ii)] $F\colon V\rightarrow \partial X \times I$ is a diffeomorphism.
\item[(iii)] $F\mid_{(M_0,\partial M_0)} = f_0$, $F\mid_{(M_1,\partial M_1)} = f_1$.
\end{enumerate}
The bundle data $\hat{F}$ is a fiberwise isomorphism of vector bundles over $F$,
$$\hat{F}\colon \nu_W \rightarrow \zeta\times I$$
with
\begin{enumerate}
\item[(iv)] $\hat{F}\mid_{\nu_{M_0}} = \hat{f}_0,\quad \hat{F}\mid_{\nu_{M_1}} = \hat{f}_1$.
\end{enumerate}
There is a map, the normal invariant, from the structure set to the normal cobordism classes of normal maps,
$$\eta \colon \Ss^{G/O}(X,\partial X)\rightarrow \mathcal{N}^{G/O}(X,\partial X).$$
It is defined as follows. Let $f\colon (M,\partial M)\rightarrow (X,\partial X)$ be a homotopy equivalence representing an element of the structure set, and let $\nu_M$ be its stable normal bundle as above. Pick a homotopy inverse $g\colon (X,\partial X)\rightarrow (M,\partial M)$ to $f$ with $\partial g = (\partial f)^{-1}$, define $\zeta = g^*(\nu_M)$ and note that $f^*(\zeta)$ is identified with $\nu_M$ since $g\circ f \simeq id_M$. The bundle map $\hat{f}\colon \nu_M\rightarrow \zeta$ is the resulting fiberwise isomorphism over $f$. The normal cobordism class of $(f,\hat{f})$ is independent of the choices. This defines the normal invariant $\eta(f)$.

The process of surgery measures to which extent $\eta$ is a bijection. The result is an exact sequence for $k+n>5$,
\begin{equation} \label{eq:(18)}
\Ss^{G/O}(D^k\times X,\partial) \stackrel{\eta}{\rightarrow} \mathcal{N}^{G/O}(D^k\times X,\partial) \stackrel{\lambda}{\rightarrow} L_{n+k}(\ZZ[\pi_1 X]) \stackrel{\alpha}{\rightarrow} \Ss^{G/O}(D^{k-1}\times X,\partial).
\end{equation}
Here $n = \dim X$, and $L_*(\ZZ[\pi_1 X])$ are Wall's $L$-groups. They are graded abelian groups that classify quadratic forms on (free) modules over the group ring $\ZZ[\pi_1 X]$ when $*$ is even and stable automorphisms of such forms when $*$ is odd. For $k>1$, the terms in \eqref{eq:(18)} are abelian groups and the maps are homomorphisms. For $k=1$, $\alpha$ is an action of $L_{n+1}(\ZZ[\pi_1 X])$ on $\Ss^{G/O}(X,\partial X)$. The set of orbits of this action is in bijection with the subset of the normal invariants that map to zero under $\lambda$.

In our applications $X$ is simply connected, and
\begin{equation} \label{eq:(19)}
L_n(\ZZ) = \left\{ \begin{array}{ll} \ZZ, & n \equiv 0 \, (4) \\ \ZZ/2, & n \equiv 2 \, (4) \\ 0, & \mbox{$n$ odd.} \end{array} \right.
\end{equation}

We next recall the important reformulation of the set of normal invariants due to Sullivan. Let $G(n)$ be the topological monoid of self homotopy equivalences of the sphere $S^{n-1}$ in the compact-open topology. It contains the orthogonal group $O(n)$. The associated homogeneous space $G(n)/O(n)$ is defined to be the homotopy theoretic fiber of the map on classifying spaces
$$\pi\colon BO(n)\rightarrow BG(n).$$
Let $*\in BG(n)$ be a basepoint. Then $G(n)/O(n)$ consists of pairs $(x,\lambda(t))$ with $x\in BO(n)$ and $\lambda(t)$ a path in $BG(n)$ subject to the conditions $\lambda(0) = \pi(x)$, $\lambda(1) = *$. Hence a map from a space $Y$ to $G(n)/O(n)$ is equivalent to a commutative diagram
$$
\xymatrix{Y \ar[r]^-f \ar[d]^-i & BO(n) \ar[d]^-\pi \\ CY \ar[r]^-g & BG(n)}
$$
where $CY$ is the cone on $Y$ and $g$ maps the cone point to $*$. Since $\pi$ maps an orthogonal vector bundle to its sphere bundle, considered as a spherical fiber space, the above diagram gives rise to an orthogonal vector bundle over $Y$ whose spherical fiber space extends over the cone, hence is trivial. In other words, $G(n)/O(n)$ classifies pairs $(\xi,t)$ of an orthogonal vector bundle $\xi$ and a fiberwise homotopy equivalence $t\colon S(\xi) \rightarrow S^{n-1}\times Y$.

There are inclusions $(G(n),O(n))\subset (G(n+1),O(n+1))$ and $G/O$ is the colimit of $G(n)/O(n)$ as $n\rightarrow \infty$. By the above, $G/O$ classifies triples $(\xi,\eta,t)$ of two stable vector bundles $\xi$ and $\eta$ over $Y$ and a homotopy equivalence $t\colon S(\xi) \rightarrow S(\eta)$. Hence the set of homotopy classes from a finite CW complex to $G/O$ can be identified with the set of isomorphism classes of triples $(\xi,\eta,t)$.

The space $G/O$ admits a multiplication, it is even an infinite loop space, so $[Y,G/O]$ is an abelian group. This structure corresponds to the following addition of triples:
$$(\xi_1,\eta_1,t_1)\oplus (\xi_2,\eta_2,t_2) = (\xi_1\oplus \xi_2,\eta_1\oplus \eta_2,t_1*t_2).$$

\begin{theorem}[Sullivan] \label{thm:3.1}
There is an isomorphism
$$\sigma\colon \mathcal{N}^{G/O}(X,\partial X) \stackrel{\cong}{\rightarrow} [X/\partial X,G/O].$$
\end{theorem}

\begin{proof}[Proof (Sketch)]
Consider a normal map
$$
\xymatrix{\nu_M \ar[d] \ar[r]^-{\hat{f}} & \zeta \ar[d] \\ M \ar[r]^-f & X}
$$
with $\nu_M$ the normal bundle of an embedding $(M^n,\partial M^n) \subset (D^{n+k},\partial D^{n+k})$ for some large $k$. View $(\nu_M,\nu_M |_\partial)$ as (relative) open tube around $(M,\partial M)$ and let
$$c_M\colon (D^{n+k},\partial D^{n+k}) \rightarrow (\Th(\nu_M),\Th(\nu_M|_{\partial M}))$$
be the associated Pontryagin-Thom collapse map into the one-point compactifications. Here the Thom space $\Th(\nu_M)$ is the one-point compactification of the total space. The collapse map $c_M$ has degree one, and we can compose with
$$(\Th(\hat{f}),\Th(\hat{f}|_\partial))\colon (\Th(\nu_M),\Th(\nu_M|_\partial)) \rightarrow (\Th(\zeta),\Th(\zeta|_{\partial X}))$$
to obtain a degree one map
$$c_M\colon (D^{n+k},\partial D^{n+k})\rightarrow (\Th(\zeta),\Th(\zeta|_{\partial X})).$$
The Pontryagin-Thom collapse map associated with an embedding $(X,\partial X)\subset (D^{n+k},\partial D^{n+k})$ yields another degree one map
$$c_X\colon (D^{n+k},\partial D^{n+k})\rightarrow (\Th(\nu_X),\Th(\nu_X|_{\partial X})).$$
The uniqueness theorem for Spivak normal fibrations \cite[Theorem I.4.19]{B} gives a fiber homotopy equivalence
$$c(f,\hat{f})\colon \dot{\nu}_X \rightarrow \dot{\zeta}$$
where $\dot{\nu}_X$ and $\dot{\zeta}$ denote the fiberwise one point compactifications, $\dot{\nu}_X = S(\nu_X\oplus \RR)$ etc. The restriction of $c(f,\hat{f})$ to $\partial X$ is the fiberwise one point compactification of a vector bundle map. Hence the triple defines the element
$$\sigma(f,\hat{f}) \in [X/\partial X, G/O]_*$$
in the based homotopy set, which in turn is equal to the unbased one since $G/O$ is simply connected. The inverse to $\sigma$ is defined by transversality.
\end{proof}

\begin{remark} \label{rmk:3.2}
There is a completely analogous theory where smooth manifolds are replaced by topological (or piece-wise linear) manifolds. The term $L_*(\ZZ[\pi_1 X])$ in \eqref{eq:(18)} is unchanged, but the normal invariant term changes from $[X/\partial X,G/O]$ to $[X/\partial X,G/\Top]$. Here $\Top$ is the union or colimit of $\Top(n)$, the group of homeomorphisms of $\RR^n$. Surgery theory in the topological category is calculational easier to handle because of the following results due to Sullivan:
\begin{align*}
[Y,G/\Top] \tensor \ZZ_{(2)} & \cong \bigoplus_{k\geq 1} \HH^{4k}(Y;\ZZ_{(2)}) \oplus \HH^{4k-2}(Y;\ZZ/2), \\
[Y,G/\Top] \tensor \ZZ[1/2] & \cong \widetilde{KO}^0(Y)\tensor \ZZ[1/2].
\end{align*}
Let us also recall another of Sullivan's results, namely that for each prime $p$
$$[Y,G/O]\tensor \ZZ_{(p)} \cong KSO^0(Y)\tensor \ZZ_{(p)} \oplus [Y, Cok J]\tensor \ZZ_{(p)}.$$
Here $\ZZ_{(p)}\subset \QQ$ is the subring of fractions with denominator prime to $p$. The homotopy groups of $Cok J$ is the largely unknown part of the stable homotopy groups of spheres that does not come from the homotopy groups of the infinite orthogonal group; $(Cok J)_\QQ \simeq *$.
\end{remark}

\subsection{Block automorphisms and Quinn's fibration}
For a smooth oriented compact manifold $X$,  $\Diff_{\partial X}(X)$ denotes the topological group of orientation preserving diffeomorphisms of $X$ (in the $C^\infty$-Whitney topology) that restrict to the identity on the boundary $\partial X$. There is an associated simplicial group, namely its singular complex, with $k$-simplices consisting of commutative diagrams
$$\xymatrix{\Delta^k\times X \ar[rr]^-\varphi \ar[dr]_-{pr_1} && \Delta^k\times X \ar[dl]^-{pr_1} \\ & \Delta^k}$$
with $\varphi$ a diffeomorphism which is the identity on $\Delta^k\times \partial X$. One may relax the condition on $\varphi$ and just assume that
\begin{equation} \label{eq:(20)}
\varphi\colon \Delta^k\times X \rightarrow \Delta^k \times X
\end{equation}
preserves the face structure of $\Delta^k$: for each face $\Delta^\ell \subset \Delta^k$, $\varphi$ maps $\Delta^\ell \times X$ into itself. Such a $\varphi$ is called a block diffeomorphism. The set of block diffeomorphisms form a $\Delta$-group or a pre-simplicial group, that is, a simplicial group without degeneracies. The general theory of $\Delta$-sets was developed in \cite{RS}. The geometric realization of a $\Delta$-set $S_\bullet$ is
$$||S_\bullet|| = \bigsqcup_k \Delta^k\times S_k/\sim,\quad (t,d_ix)\sim (d^it,x)$$
for $t\in \Delta^{k-1}$, $x\in S_k$. $\Delta$-groups are fibrant (satisfy a Kan condition). This implies that
\begin{equation} \label{eq:(36)}
\pi_k\tDiff_{\partial X}(X) = \pi_0\Diff_{\partial(D^k\times X)}(D^k\times X).
\end{equation}
We remarked above that the singular complex $\Sing_\bullet \Diff_{\partial X}(X)$ is a subgroup of $\tDiff_{\partial X}(X)$. In the rest of the paper we shall replace $\Diff_{\partial X}(X)$ by its singular complex without so indicating in the notation. (Alternatively one could pass to geometric realizations as the realization of $\Sing_\bullet\Diff_{\partial X}(X)$ is equivalent to $\Diff_{\partial X}(X)$ in the $C^\infty$-Whitney topology.) Note that the $1$-simplices of $\tDiff_{\partial X}(X)$ are diffeomorphisms of the cylinder $I\times X$ that preserve top and bottom and is the identity on $I\times \partial X$; the $1$-simplices of $\Diff_{\partial X}(X)$ further preserve the level $\{t\} \times X$ for each $t\in I$.

In \eqref{eq:(20)} we may use face preserving homotopy equivalences instead of diffeomorphisms to obtain a $\Delta$-monoid denoted $\tAut_{\partial X}(X)$. Up to homotopy there is no difference between fiberwise homotopy equivalences of fiber bundles and homotopy equivalences of their total spaces, cf. \cite[Theorem 6.1]{Dold}. This implies that the submonoid of fiberwise homotopy equivalences,
\begin{equation} \label{eq:(22)}
\xymatrix{\Delta^k\times X \ar[dr] \ar[rr]^-f && \Delta^k\times X \ar[ld] \\ & \Delta^k}
\end{equation}
is homotopy equivalent to $\tAut_{\partial X}(X)$. The submonoid defined by \eqref{eq:(22)} is the singular complex of $\aut_{\partial X}(X)$, the topological monoid of homotopy self-equivalences of $X$ with compact-open topology, so
$$\tAut_{\partial X}(X) \simeq \aut_{\partial X}(X),$$
in the category of $\Delta$-monoids. Again we will tacitly replace $\aut_{\partial X}(X)$ with its singular complex without change of notation. In contrast $\tDiff_{\partial X}(X)$ is very different from $\Diff_{\partial X}(X)$. Indeed, the homogeneous space $\tDiff_{\partial X}(X)/\Diff_{\partial X}(X)$ is closely related to Waldhausen's functor $A(X)$, as explained in the introduction.

The homotopy theoretic fiber of the obvious map
$\widetilde{J}\colon B\tDiff_{\partial X}(X) \rightarrow B\tAut_{\partial X}(X)$ is by definition the homogeneous space $\tAut_{\partial X}(X)/\tDiff_{\partial X}(X)$. Below we compare it with the structure space $\Ssp_{\partial X}^{G/O}(X)$ defined by F. Quinn in \cite{Q}. Its homotopy groups are
$$\pi_k \Ssp_{\partial X}^{G/O}(X) = \Ss_{\partial(D^k\times X)}^{G/O}(D^k\times X).$$
Moreover, Quinn constructs a homotopy fibration
\begin{equation} \label{eq:(23)}
\Ssp_{\partial X}^{G/O}(X) \rightarrow \Map_*(X/\partial X,G/O)\rightarrow \Lsp(X)
\end{equation}
whose homotopy exact sequence is the surgery exact sequence provided $\dim X \geq 5$. The space $\Lsp(X)$ has homotopy groups
$$\pi_i \Lsp(X) = L_{i+\dim X}(\ZZ[\pi_1 X]).$$
The structure space is defined simplicially. Its $k$-simplices are pairs $(W,f)$ where $W$ is a manifold with corners with $\dim W = k + \dim X$, and $f\colon W\rightarrow \Delta^k\times X$ is a simple homotopy equivalence which restricts to a diffeomorphism on the boundary faces ($W$ is a $(k+2)$-ad in the terminology of \cite{W}). There is an obvious map
$$s_X\colon \tAut_{\partial X}(X)/\tDiff_{\partial X}(X) \rightarrow \Ssp_{\partial X}^{G/O}(X).$$
It is a consequence of the $s$-cobordism theorem that $s_X$ defines a weak homotopy equivalence of the connected components of the identity. Indeed, the homotopy groups of the target are represented by diagrams
$$\xymatrix{W \ar[r]^-g & D^k\times X \\ \partial W \ar[r]^-{\partial g} \ar@{^{(}->}[u] & \partial(D^k\times X) \ar@{^{(}->}[u]}$$
with $g$ a simple homotopy equivalence and $\partial g$ a diffeomorphism. For $k=1$, $W$ is a manifold with
$$\partial W = \partial_0W \cup V \cup \partial_1 W,\quad \partial V = \partial(\partial_0 W) \sqcup \partial(\partial_1 W)$$
where $\partial_\nu W = (\partial g)^{-1}(\{\nu\}\times X)$ and $V = (\partial g)^{-1}(I\times \partial X)$. Thus $(W,V)$ is a relative $h$-cobordism and since $g$ is assumed to be a simple homotopy equivalence $W$ is a relative $s$-cobordism. This implies by the $s$-cobordism theorem a diffeomorphism
$$(W,V,\partial_0 W,\partial_1 W) \cong (I\times X,I\times \partial X, 0\times X, 1\times X).$$
For $k>1$, write $D^k = I\times D^{k-1}$ and use the above with $X$ replaced by $D^{k-1}\times X$. This shows that
$$s_X \colon \pi_k(\tAut_{\partial X}(X)/\tDiff_{\partial X}(X), 1_X) \rightarrow \pi_k(\Ssp_{\partial X}^{G/O}(X),1_X)$$
is surjective. It is clearly injective and hence an isomorphism, so that $s_X$ defines a weak homotopy equivalence of components of the identity
\begin{equation} \label{eq:(24)}
s_X\colon \big[ \tAut_{\partial X}(X)/\tDiff_{\partial X}(X) \big]_{(1)} \stackrel{\sim}{\rightarrow} \Ssp_{\partial X}^{G/O}(X)_{(1)}.
\end{equation}

In our applications $X = N =  M_g^{2d} \setminus \interior D^{2d}$ is simply connected, and since $L_{2k+1}(\ZZ) = 0$ we can reformulate $\eqref{eq:(23)}$ to the homotopy fibration
\begin{equation} \label{eq:(25)}
\big[ \tAut_{\partial N}(N)/\tDiff_{\partial N}(N) \big]_{(1)} \rightarrow \Map_*(M_g,G/O)_{(1)} \rightarrow \Lsp(M_g)_{(1)}.
\end{equation}

The lemma below is needed for calculations in the next section. Let
$$f\colon X_1\rightarrow X_2,\quad g\colon X_2\rightarrow X_3$$
be simple homotopy equivalences of smooth manifolds with $\partial f$ and $\partial g$ diffeomorphisms. They define elements of $\Ss^{G/O}(X_2,\partial X_2)$ and $\Ss^{G/O}(X_3,\partial X_3)$, respectively.

\begin{lemma} \label{lemma:norminv}
In $[X_3/\partial X_3,G/O]$ we have the formula
$$\eta(g\circ f) = (g^*)^{-1}(\eta(f)) + \eta(g).$$
\end{lemma}

\begin{proof}
We use the description of the normal invariant from \S \ref{sec:surgery}. Let $\nu_i$ be the stable normal bundle of $X_i$ and let
$$\alpha_i \colon S^{n+k} \rightarrow \Th(\nu_i)/\Th(\nu_i\mid_\partial) =: \Th^{rel}(\nu_i)$$
be the Pontryagin-Thom collapse map, where $n = \dim X_i$ and $k = \dim \nu_i \gg n$. Let $\overline{f}\colon X_2\rightarrow X_1$ and $\overline{g}\colon X_3\rightarrow X_2$ be homotopy inverses to $f$ and $g$, and let
$$c(f)\colon \dot{\nu}_2 \rightarrow \overline{f}^*(\dot{\nu}_1)\quad c(g)\colon \dot{\nu}_3\rightarrow \overline{g}^*(\dot{\nu}_2)$$
be the fiberwise one-point compactifications, which define $\eta(f)$ and $\eta(g)$. The diagram
$$
\xymatrix{S^{n+k} \ar[d]^-{\alpha_1} \ar[r]^-{\alpha_2} & \Th^{rel}(\nu_2) \ar[d]^-{c(f)} \\ \Th^{rel}(\nu_1) \ar[r]^-{f_*} & \Th^{rel}(\overline{f}^*(\nu_1))}
$$
is homotopy commutative, and similarly for $c(g)$. Now consider the homotopy commutative diagram
$$
\xymatrix{& S^{n+k} \ar[d]^-{\alpha_2} \ar[r]^-{\alpha_1} & \Th^{rel}(\nu_1) \ar[d]^-{f_*} \\
S^{n+k} \ar@{=}[ur] \ar[d]^-{\alpha_3} \ar[r]^-{\alpha_2} & \Th^{rel}(\nu_2) \ar[d]^-{g_*} \ar[r]^-{c(f)} & \Th^{rel}(\overline{f}^*(\nu_1)) \ar[d]^-{g_*} \\
\Th^{rel}(\nu_3) \ar[r]^-{c(g)} & \Th^{rel}(\overline{g}^*(\nu_2)) \ar[r]^-{\overline{g}^*(c(f))} & \Th^{rel}(\overline{g}^*\overline{f}^*(\nu_1))}
$$
The outer diagram shows that
$$\overline{g}^*(c(f))\circ c(g) \simeq c(g\circ f) \colon \dot{\nu}_1 \rightarrow \overline{g}^*\overline{f}^*(\dot{\nu}_1)$$
and hence the result.
\end{proof}

\subsection{Homological stability of block diffeomorphisms}
The main result of this section is the following analog of Theorem \ref{thm:main}. Its proof occupies the entire section.
\begin{theorem} \label{thm:main-block}
For $M_g = \#^g(S^d\times S^d)$ with $d>2$, the stabilization map
$$\HH_k(B\tDiff_D(M_g);\QQ) \rightarrow \HH_k(B\tDiff_D(M_{g+1});\QQ)$$
is an isomorphism in the range $k<\min(d-2,\frac{1}{2}(g-5))$.
\end{theorem}
We first use the fibration \eqref{eq:(25)} to calculate the homology and homotopy groups of
$$\tAut_{\partial N}(N)/\tDiff_{\partial N}(N), \quad N = M_g\setminus \interior D,$$
and then the rational homotopy theory of \S2 to finish the proof of Theorem \ref{thm:main-block}.

The rational homotopy type of $G/O$ is
$$(G/O)_\QQ \simeq BO_\QQ \simeq \prod_{\ell \geq 1} K(\QQ,4\ell),$$
so that
\begin{equation} \label{eq:(41)}
\pi_k \Map_*(M_g,G/O) \tensor \QQ \cong \HH^d(M_g;\QQ)\tensor \pi_{k+d}(G/O) \oplus \HH^{2d}(M_g;\QQ)\tensor \pi_{k+2d}(G/O).
\end{equation}

\begin{lemma} \label{lemma:calc}
For $k+2d \equiv 0 \,\, (4)$ the map
$$\lambda_*\colon \pi_k\Map_*(M_g,G/O)\tensor \QQ \rightarrow L_{k+2d}(\ZZ)\tensor \QQ$$
restricts to an isomorphism of the second summand in \eqref{eq:(41)}.
\end{lemma}

\begin{proof}
We compare the smooth surgery exact sequence with its version for topological manifolds, cf. Remark \ref{rmk:3.2}. In the diagram
$$
\xymatrix{\pi_k\Map_*(M_g,G/O)\tensor \QQ \ar[r]^-{\lambda_*} \ar[d] & L_{k+2d}(\ZZ)\tensor \QQ \ar@{=}[d] \\ \pi_k\Map_*(M_g,G/Top)\tensor \QQ \ar[r]^-{\lambda_*} & L_{k+2d}(\ZZ)\tensor \QQ}
$$
the left hand vertical map is an isomorphism because $\pi_k(Top/O)$ is finite by \cite{KM,KS}.

Milnor's plumbing construction yields a smooth normal map $(f,\hat{f})$
$$f\colon (K^n,\partial K^n) \rightarrow (D^n,S^{n-1}),$$
with $\partial f$ a homotopy equivalence, and $\lambda_*(f,\hat{f}) \ne 0$ for $n\equiv 0$ (mod $2$), cf. \cite{B}. Since any smooth homotopy sphere is homeomorphic to the standard sphere,
$$\partial f\colon \partial K^n \rightarrow S^{n-1}$$
can be assumed to be a homeomorphism. For $n=k+2d \equiv 0$ (mod $4$), the normal map $(g,\hat{g}) = (id\cup f,id\cup \hat{f})$,
$$g\colon M_g^{2d}\times D^k \setminus \interior(D^{2d}\times D^k) \cup_\partial K^{k+2d} \rightarrow M_g^{2d}\times D^k,$$
has $\lambda_*(g,\hat{g}) = \lambda_*(f,\hat{f}) \ne 0$.
\end{proof}

The homotopy exact sequence of \eqref{eq:(25)} with $N = M_g^{2d}\setminus \interior(D^{2d})$
$$
[M_g^{2d},\Omega^{k+1}G/O]_* \stackrel{\eta_{k+1}}{\rightarrow} L_{2d+k+1}(\ZZ) \rightarrow \pi_k(\tAut_{\partial N}(N)/\tDiff_{\partial N}(N),id_N) \stackrel{\eta_k}{\rightarrow}\cdots
$$
induces short exact sequences
$$0\rightarrow L_{2d+k+1}(\ZZ)/\im \eta_{k+1} \rightarrow \pi_k(\tAut_{\partial N}(N)/\tDiff_{\partial N}(N)) \rightarrow \im \eta_k \rightarrow 0.$$
The left-hand term is a finite cyclic group and
$$\im \eta_k \tensor \QQ \cong \HH^d(M_g;\QQ)\tensor \pi_{k+d}(G/O)$$
according to Lemma \ref{lemma:calc}. The fundamental group of $\tAut_{\partial N}(N)/\tDiff_{\partial N}(N)$ is meta-abelian with finite kernel, and maybe abelian. Let us write
$$\pi_k := \pi_k(\tAut_{\partial N}(N)/\tDiff_{\partial N}(N),id_N)$$
and let $\pi_k^{ab}$ be the abelian quotient.

\begin{theorem} \label{thm:fg}
For $d>2$ and $N = M_g^{2d}\setminus \interior(D^{2d})$,
\begin{enumerate}
\item[(i)] $\pi_k^{ab}\tensor \QQ \cong \HH^d(M_g;\QQ)\tensor \pi_{k+d}(G/O).$
\item[(ii)] $\HH_*( \big[ \tAut_{\partial N}(N)/\tDiff_{\partial N}(N) \big]_{(1)};\QQ) \cong \Lambda(\pi_*^{ab}\tensor \QQ)$, the free graded commutative algebra on the graded vector space $\pi_*^{ab}\tensor \QQ$.
\end{enumerate}
\end{theorem}

\begin{proof}
(i) follows from the previous lemma and \eqref{eq:(25)}. For (ii) we first note that since $L_n(\ZZ) = 0$ for $n$ odd,
\begin{equation} \label{eq:(27)}
\pi_k (\tAut_{\partial N}(N)/\tDiff_{\partial N}(N),1_N)\tensor \QQ \rightarrow \pi_k \Map_*(M_g,G/O)\tensor \QQ
\end{equation}
is injective. The space $G/O$ is an infinite loop space, in particular a loop space, $G/O \simeq \Omega(B(G/O))$. The same is then the case for $\Map_*(M_g,G/O)$. A loop space $X = \Omega Y$ has no rational $k$-invariants; its rational homotopy type is a product of Eilenberg-MacLane spaces. This happens if and only if the following criterion is satisfied: For each $\alpha \in \pi_r(X)\tensor \QQ$ there exists a cohomology class $\xi\in \HH^r(X;\QQ)$ with $\langle \xi , \alpha \rangle \ne 0$. Since \eqref{eq:(27)} is injective, the criterion is satisfied for $\big[ \tAut_{\partial N}(N)/\tDiff_{\partial N}(N) \big]_{(1)}$ which therefore has the rational homotopy type of a product of Eilenberg-MacLane spaces. For such spaces, rational homology and homotopy are related as stated in (ii).
\end{proof}

We shall next examine the homomorphism
$$\widetilde{J}\colon \pi_0 \tDiff_{\partial N}(N) \rightarrow \pi_0 \tAut_{\partial N}(N)$$
when $N = M_g\setminus \interior D$. The combination of Theorem \ref{thm:arithmetic} and \eqref{eq:red} provides the exact sequence
$$0\rightarrow K \rightarrow \pi_0\aut_{\partial N}(N) \rightarrow \Aut(\ZZ^{2g},q) \rightarrow 0$$
with $K$ finite, and we remember that $\aut_{\partial N}(N) \simeq \tAut_{\partial N}(N)$.

We next involve C.T.C. Wall's classification of $2d$-dimensional $(d-1)$-connected manifolds \cite{Wall,Wall3}. See also Kreck's paper \cite{K}. Given such a manifold $M$ (e.g. $M_g$ above), write $N = M\setminus \interior D$. Wall defines a quadratic map
$$\alpha\colon \HH_d(N;\ZZ) \rightarrow \ZZ/(1-(-1)^d)\ZZ$$
with
$$\alpha(x+y) - \alpha(x) - \alpha(y) = q(x,y),$$
where $q$ is the intersection pairing. We recall the definition. A homology class $x\in\HH_d(N)$ is represented by an embedding $x\colon S^d\hookrightarrow M$ with a stably trivial normal bundle $\nu(x)$. Its isomorphism class $[\nu(x)]$ is an element
$$[\nu(x)]\in \ker(\pi_d(BSO(d))\stackrel{i_*}{\rightarrow} \pi_d(BSO(d+1))).$$
The calculation of $\ker(i_*)$, contained in \cite{Steenrod,B} is
$$\ker(i_*) = \ZZ/(1-(-1)^d)\ZZ$$
when $d\ne 1,3,7$. In the exceptional cases, $\ker(i_*) = 0$. For even $d$ or in the exceptional cases, $\alpha(x)$ is determined by $q(x,x)$. Wall shows that
$$\widetilde{h}\colon \tDiff_{\partial N}(N) \rightarrow \Aut(\HH_d(M),q,\alpha)$$
is surjective. The target of $\widetilde{h}$ is the automorphism group of the quadratic form $(\HH_d(M),q,\alpha)$. Its kernel is examined in \cite{KS,Wall}.

We now return to the situation $M = M_g^{2d} = \#^g(S^d\times S^d)$, $N = M_g\setminus \interior D$ and
$$\widetilde{J}\colon \pi_0 \tDiff_{\partial N}(N) \rightarrow \pi_0 \tAut_{\partial N}(N).$$

\begin{lemma} \label{lemma:tilde}
There is an exact sequence of groups
$$1\rightarrow \widetilde{K}_g \rightarrow \im(\pi_0 \widetilde{J}) \rightarrow \Aut(\HH_d(M_g),q,\alpha) \rightarrow 1$$
with $\widetilde{K}_g$ finite.
\end{lemma}

\begin{proof}
By \eqref{eq:red}, Theorem \ref{thm:arithmetic} and the above discussion of \cite{Wall}, we have a commutative diagram of groups with exact rows

$$
\xymatrix{1 \ar[r] & L \ar[d] \ar[r] & \pi_0\tDiff_{\partial N}(N) \ar[d] \ar[r]^-{\widetilde{h}} & \Aut(\HH_d(M_g),q,\alpha) \ar@{^{(}->}[d] \ar[r] & 1 \\
1 \ar[r] & K \ar[r] & \pi_0\tAut_{\partial N}(N) \ar[r]^-h & \Aut(\HH_d(M_g),q) \ar[r] & 1.}
$$
The result follows.
\end{proof}

For a finitely generated free abelian group $A$ and $\epsilon\in \{-1,1\}$, let $H(A,\epsilon)$ denote the hyperbolic module
$$H(A,\epsilon) = A\oplus A^*$$
with quadratic form
$$\alpha \colon H(A,\epsilon)\rightarrow \ZZ/(1-\epsilon)\ZZ$$
and associated bilinear form
$$\mu \colon H(A,\epsilon)\times H(A,\epsilon) \rightarrow \ZZ$$
given by
$$\alpha(x,f) = f(x),\quad \mu((x,f),(y,g)) = f(y) + \epsilon g(x).$$
If $e_1,\ldots,e_g$ is a basis for $A$ and $f_1,\ldots,f_g$ the dual basis for $A^*$, then
$$\mu\colon \ZZ^{2g}\times \ZZ^{2g}\rightarrow \ZZ$$
is given by
$$\mu(u,v) = u^t\left( \begin{array}{cc} 0 & I \\ \epsilon I & 0 \end{array} \right)v, \quad u,v\in \ZZ^{2g}.$$
We have
$$\Aut(H(\ZZ^g,\epsilon),\mu) = \left\{ \begin{array}{ll} \OO_{g,g}(\ZZ), & \epsilon = +1 \\ \Sp_{2g}(\ZZ), & \epsilon = -1 \end{array} \right.$$
Notice that
$$\Aut(H(\ZZ^g,+1),\mu,\alpha) = \Aut(H(\ZZ^g,+1),\mu),$$
whereas $\Aut(H(\ZZ^g,-1),\mu,\alpha)$ is a proper subgroup of $\Aut(H(\ZZ^g,-1),\mu)$ of $2$-power index.

For $M_g = \#^g (S^d\times S^d)$, the quadratic module $(\HH_d(M_g),q,\alpha)$ is hyperbolic,
$$(\HH_d(M_g),q,\alpha) \cong H(\ZZ^g,(-1)^d).$$

The subgroup $\im(\pi_0 \widetilde{J})$ of $\pi_0\tAut_{\partial N}(N)$ has finite index by Theorem \ref{thm:arithmetic} and Lemma \ref{lemma:tilde}, so it defines a finite covering
$$\rho \colon \overline{B}\tAut_{\partial N}(N)\rightarrow B\tAut_{\partial N}(N)$$
and the map $B\widetilde{J}$ lifts to a map
$$\pi\colon B\tDiff_{\partial N}(N) \rightarrow \overline{B} \tAut_{\partial N}(N).$$
Its homotopy fiber is easily identified with the identity component of the homogeneous space
$$\tAut_{\partial N}(N)/\tDiff_{\partial N}(N).$$
The resulting Serre spectral sequence has

\begin{equation} \label{eq:(43)}
E_{p,q}^2(M_g) = \HH_p(\overline{B}\tAut_{\partial N}(N);\HH_q( \big[ \tAut_{\partial N}(N)/\tDiff_{\partial N}(N) \big]_{(1)};\QQ))
\end{equation}
and it converges to
$$\HH_*(B\tDiff_{\partial N}(N);\QQ) = \HH_*(B\tDiff_D(M_g);\QQ).$$
There are twisted coefficients in \eqref{eq:(43)} in the sense that $\pi_1\overline{B}\tAut_{\partial N}(N) = \im(\pi_0(\widetilde{J})) \subset \pi_0\tAut_{\partial N}(N)$ acts on the fiber. The action can be described as follows.

Let $\varphi\colon M_g\rightarrow M_g$, $\varphi|_D = id$ represent an element of $\pi_0\tDiff_D(M_g)$ and let
$$f\colon D^k\times N \rightarrow D^k\times N$$
be a homotopy equivalence with $\partial f$ a diffeomorphism; $f$ represents an element of the structure set
$$\Ss_{\partial(D^k\times N)}(D^k\times N) = \pi_k(\tAut_{\partial N}(N)/\tDiff_{\partial N}(N),1_N)$$
and $[\varphi]\in \im(\pi_0 \widetilde{J})$ acts on $[f]$ by composition
$$\xymatrix{D^k \times N \ar[r]^-f & D^k\times N \ar[r]^-{D^k\times \varphi} & D^k\times N.}$$
Since $\varphi$ is a diffeomorphism, Lemma \ref{lemma:norminv} reduces to
$$\eta((D^k\times \varphi)\circ f) = (D^k\times \varphi)^{-1}(\eta(f)).$$

It follows from Theorem \ref{thm:fg}(i) that the action on the $k$'th rational homotopy group
\begin{equation} \label{eq:(29)}
\pi_k(\tAut_{\partial N}(N)/\tDiff_{\partial N}(N),1_N)\tensor \QQ = \HH^d(M_g;\QQ)\tensor \pi_{k+d}(G/O)
\end{equation}
is via the induced map $(\varphi^{-1})^*\tensor id$. The action on homology is $\Lambda((\varphi^{-1})^*\tensor id)$ according to Corollary \ref{thm:fg}(ii).

It follows that in \eqref{eq:(43)}, the action of $\pi_1\overline{B}\tAut_{\partial N}(N) = \im(\pi_0 \widetilde{J})$ on the fiber is via the projection
$$\widetilde{h}\colon \im(\pi_0\widetilde{J}) \rightarrow \Aut(H(\ZZ^g,(-1)^d),q,\alpha)$$
and this action is understood homologically by results from \cite{C}. More precisely, Theorem \ref{thm:aut} tells us that
$$\HH_*(B\tAut_{\partial N}(N);\QQ) \rightarrow \HH_*(B\pi_0(\tAut_{\partial N}(N));\QQ)$$
is $d$-connected. The same is then the case for
$$\HH_*(\overline{B}\tAut_{\partial N}(N);\QQ) \rightarrow \HH_*(B\im(\pi_0 \widetilde{J});\QQ).$$
For the $E^2$-term in \eqref{eq:(43)} we get for $p<d-1$
$$E_{p,q}^2(M_g) \cong \HH_p(B\im(\pi_0\widetilde{J}); \HH_q( \big[ \tAut_{\partial N}(N)/\tDiff_{\partial N}(N) \big]_{(1)};\QQ)),$$
and since $\widetilde{h}$ has finite kernel by Lemma \ref{lemma:tilde} and the action is through $\widetilde{h}$, we get for $p<d-1$
\begin{equation} \label{eq:e2}
E_{p,q}^2 \cong \HH_p(\Aut(H(\ZZ^g,(-1)^d),q,\alpha);\Lambda(\pi_*^{ab}\tensor \QQ))
\end{equation}
with $\Lambda(\pi_*^{ab}\tensor \QQ)$ displayed in Theorem \ref{thm:fg}.

\begin{theorem}[{\cite{C}}] \label{thm:charney}
For $k\geq 1$, the stabilization map
$$\HH_p(\Aut(H(\ZZ^g,\epsilon),q,\alpha),H(\ZZ^g,\epsilon)^{\tensor k}) \rightarrow \HH_p(\Aut(H(\ZZ^{g+1},\epsilon),q,\alpha);H(\ZZ^{g+1},\epsilon)^{\tensor k})$$
is $\frac{1}{2}(g-4-k)$-connected.  $\hfill \square$
\end{theorem}

\begin{proof}[Proof of Theorem {\ref{thm:main-block}}]
We use the Serre spectral sequence with $E^2$-term \eqref{eq:(43)} and abutment
$$\HH_*(B\tDiff_{\partial N}(N);\QQ) = \HH_*(B\tDiff_D(M_g);\QQ).$$
For base degree $p<d-1$ we proved above that
$$E_{p,*}^2 \cong \HH_p(\Aut(H(\ZZ^g,\epsilon);q,\alpha);\Lambda(\pi_*^{ab}\tensor \QQ))$$
with
$$\pi_*^{ab} \cong \HH_d(M_g;\QQ)\tensor \pi_{*+d}(G/O).$$
Theorem \ref{thm:charney} implies that
$$E_{p,q}^2(M_g)\rightarrow E_{p,q}^2(M_{g+1})$$
is an isomorphism in the stated range of (total) degrees. The same is then true for the abutment.
\end{proof}

\section{Homological stability for $B\Diff_D(M_g)$}
The passage from the group of block diffeomorphisms to the group of actual diffeomorphisms is a consequence of Morlet's lemma of disjunction, which we now recall in the form given in \cite{BLR}.

Let $V$ be a compact $n$-manifold (possibly with boundary) and let $D_0\subset \operatorname{int}(V)$ be an $n$-disk. There is a diagram of inclusions
$$
\xymatrix{\Diff_\partial(D_0) \ar[d] \ar[r] & \Diff_\partial(V) \ar[d] \\ \tDiff_\partial(D_0) \ar[r] & \tDiff_\partial(V)}
$$
where the horizontal maps extend a diffeomorphism of the disk $D_0$ by the identity in the complement. We consider the induced diagram
\begin{equation} \label{5.1}
\xymatrix{B\Diff_\partial(D_0) \ar[d] \ar[r] & B\Diff_\partial(V) \ar[d] \\ B\tDiff_\partial(D_0) \ar[r] & B\tDiff_\partial(V)}
\end{equation}
Morlet's lemma of disjunction is the following result about the horizontal homotopy fibers in \eqref{5.1}, see e.g., \cite[p.31]{BLR}.
\begin{theorem}[Morlet] \label{thm:5.1}
If $V$ is $k$-connected and $k+1 < \frac{1}{2}\dim V$, then
$$\pi_j(\Diff_\partial(V),\Diff_\partial(D_0))\rightarrow \pi_j(\tDiff_\partial(V),\tDiff_\partial(D_0))$$
is $(2k-2)$-connected.
\end{theorem}
It follows from \eqref{5.1} that the vertical fibers are also related by a $(2k-2)$-connected map. For $N=M_g\setminus \operatorname{int} D$, we get that
$$\pi_j(\tDiff_\partial(D_0)/\Diff_\partial(D_0)) \rightarrow \pi_j(\tDiff_D(M_g)/\Diff_D(M_g))$$
is $2(d-2)$-connected which in turn implies that
$$\pi_j(\tDiff_D(M_g)/\Diff_D(M_g)) \rightarrow \pi_j(\tDiff_D(M_{g+1})/\Diff_D(M_{g+1}))$$
is $2(d-2)$-connected.

The combination of Theorem \ref{thm:main-block} and Morlet's theorem proves our main result.
\begin{theorem} \label{thm:5.2}
For $d>2$ the stabilization homomorphism
$$\HH_k(B\Diff_D(M_g^{2d});\QQ)\rightarrow \HH_k(B\Diff_D(M_{g+1}^{2d});\QQ)$$
is an isomorphism for $k<\min(d-2,\frac{1}{2}(g-5))$.
\end{theorem}

For oriented surfaces ($d=1$), the corresponding stability theorem has range $k<\frac{1}{3}(2g-1)$. Moreover, the forgetful map
\begin{equation} \label{5.2}
\HH_*(B\Diff_D(M_g^2))\rightarrow \HH_*(B\Diff(M_g^2))
\end{equation}
has a similar range of stability. For $d>2$ there is no analog of \eqref{5.2}. The isotopy extension theorem implies that
\begin{equation} \label{5.3}
\Diff_D(M_g) \rightarrow \Diff(M_g) \rightarrow \Emb(D,M_g)
\end{equation}
is a Serre fibration. Use of an exponential function shows that
$$\Emb(D,M_g) \simeq \Fr(TM_g),$$
the oriented frame bundle of the tangent bundle,
$$\SO(2d)\rightarrow \Fr(TM_g) \rightarrow M_g.$$
Since $M_g$ is $(d-1)$-connected the fiber in
\begin{equation} \label{eq:(32)}
\Fr(TM_g)\rightarrow B\Diff_D(M_g) \stackrel{f}{\rightarrow} B\Diff(M_g)
\end{equation}
has the same $(d-1)$-type as $\SO(2d)$. The non-zero homotopy groups of $\SO(2d)$ prevents a stability range for $f$, even rationally.

However, given the proof of Theorem 1.1 from \cite{GR-W} it is not to be expected that there would be a stability range for
$$\HH_*(B\Diff_D(M_g^{2d}))\rightarrow \HH_*(B\Diff(M_g^{2d})).$$
Here are a few words of explanation.

For $M = M_g^{2d}$, consider the spaces of bundle maps
$$\Bun^\partial(TM,\theta^*(U_{2d})) \subset \Bun(TM,\theta^*(U_{2d}))$$
where $\Bun(TM,\theta^*(U_{2d}))$ is the space of fiberwise isomorphisms
$$\xymatrix{TM \ar[d] \ar[r]^-{\hat{\hat{\tau}}} & \theta_{d+1}^*(U_{2d}) \ar[d] \\ M \ar[r]^-{\hat{\tau}} & BO(2d)[d+1,\infty)}$$
where $\hat{\tau}$ is a lifting of the tangent bundle map over $\theta_{d+1}$, cf. Section \ref{sec:intro} for notation. $\Bun^\partial(TM,\theta_{d+1}^*U_{2d})$ is the subspace where $\hat{\tau}$ and $\hat{\hat{\tau}}$ is fixed on $TM\mid_{d^{2d}}$. The diffeomorphism group $\Diff_D(M)$ acts on the subspace, and Theorem 1.1 really about the Borel (or stack) quotients
$$\Bun^\delta(TM,\theta_{d+1}^* U_{2d}) \dquot \Diff_D(M) = E\Diff_D(M)\times_{\Diff_D(M)} \Bun^\delta(TM,\theta_{d+1}^* U_{2d}).$$
But $\Bun^\delta(TM,\theta_{d+1}^* U_{2d})$ is contractible by easy obstruction theory, so
$$B\Diff_D(M) \simeq \Bun^\delta(TM,\theta_{d+1}^* U_{2d}) \dquot \Diff_D(M).$$
In contrast, the space $\Bun(TM,\theta_{d+1}^*U_{2d})$ is not contractible, so the stability theorem to be expected is that
$$\HH_*(\Bun^\partial(TM,\theta_{d+1}^*U_{2d}) \dquot \Diff_D(M);\QQ) \rightarrow \HH_*(\Bun(TM,\theta_{d+1}^*U_{2d}) \dquot \Diff(M);\QQ)$$
is an isomorphism in a range of dimensions, and this is in agreement with \eqref{eq:(32)}.

The use of Morlet's lemma of disjunction prevents the methods of this paper to improve the stability range for $\HH_*(B\Diff_D(M_g);\QQ)$ beyond $*<2(d-2)$. However, the stability range for $\HH_*(B\tDiff_D(M_g);\QQ)$ can be improved through a deeper analysis, involving more rational homotopy theory. We will return to this question in a sequel to this paper.

\end{document}